\documentclass[11pt]{amsart}
\usepackage[a4paper,margin=29mm]{geometry}
\usepackage[numbers,sort&compress]{natbib}
\usepackage[T1]{fontenc}
\usepackage{lmodern}
\usepackage{amsmath,amssymb,amsthm,mathtools}
\usepackage{booktabs}
\usepackage{microtype,needspace}
\newtheorem{theorem}{Theorem}[section]
\newtheorem{lemma}[theorem]{Lemma}
\newtheorem{proposition}[theorem]{Proposition}
\newtheorem{corollary}[theorem]{Corollary}
\newtheorem*{theoremA}{Theorem A}
\newtheorem*{theoremB}{Theorem B}
\newtheorem*{theoremC}{Theorem C}
\newtheorem*{theoremD}{Theorem D}
\theoremstyle{definition}
\newtheorem{definition}[theorem]{Definition}
\newtheorem{remark}[theorem]{Remark}

\numberwithin{equation}{section}
\newcommand{\F}{\mathbf F_p}
\newcommand{\Z}{\mathbf Z}
\newcommand{\Q}{\mathbf Q}
\newcommand{\p}{\mathfrak p}
\newcommand{\Aop}{\mathcal A}
\newcommand{\Ltwo}{\mathcal L_2}
\newcommand{\Lthree}{\mathcal L_3}
\newcommand{\Irr}{\mathcal I_p}
\newcommand{\Irru}{\mathcal I_p^{u}}
\newcommand{\Ehat}{\widehat E}
\DeclareMathOperator{\ord}{ord}
\DeclareMathOperator{\Cl}{Cl}
\DeclareMathOperator{\Gal}{Gal}
\newcommand{\Nm}{\operatorname{N}_{K/\Q}}
\usepackage[hidelinks]{hyperref}
\usepackage{xurl}
\hypersetup{pdftitle={Bernoulli numbers and the integer group determinants of Cp x Cp},pdfauthor={Chatchawan Panraksa}}
\title[Bernoulli numbers and integer group determinants]{Bernoulli numbers and the integer group determinants of $C_p\times C_p$}
\author{Chatchawan Panraksa}
\address{Applied Mathematics Program, Mahidol University International College, Salaya, Nakhon Pathom 73170, Thailand}
\email{chatchawan.pan@mahidol.ac.th}
\thanks{Corresponding author.}
\date{}
\begin{document}
\begin{abstract}
Let $e(p)$ be the least positive integer $e$ such that every multiple of
$p^e$ occurs as an integer group determinant of $C_p\times C_p$.
For primes $p\ge7$, the residues modulo $p$ of the global units of
$\Z[\zeta_p]$ determine $e(p)$, and $e(p)\le\lfloor7(p+1)/4\rfloor$.
Under Vandiver's conjecture, verified for $p<2^{31}$, we have
$e(p)=p+3+2r_p$, where $r_p$ is the least positive $r$ satisfying
$(2r+1)i>r(p-1)$ for each irregular pair $(p,2i)$ with $r\mid i$,
that is, each $1\le i\le(p-3)/2$ with $r\mid i$ for which $p$ divides
the numerator of $B_{2i}$.
In particular, $e(p)=p+5$ for regular $p$. A bound on the cofactor
gives a sufficient condition for $p^{e(p)}$ to be the least positive
$p$-divisible determinant. It holds for every prime $11\le p<10^7$.
For $p=103$, the first prime with $e(p)>p+5$, this minimum is $103^{116}$.
For regular $p$, principal ideals and affine systems over $\mathbf F_p$
describe the determinant set. For regular $p\ge11$, a prime $\ell\ne p$
occurs as a cofactor at valuation $p+4$ precisely when it splits
completely in $\Q(\zeta_p)$ and a prime ideal above it is principal
with all odd logarithmic coordinates of degrees $3$ through $p-6$
zero and at least one of the last two, of degrees $p-4$ and $p-2$, nonzero. These primes have natural density $(p+1)/(h\,p^{(p-3)/2})$,
where $h$ is the class number of $\Q(\zeta_p)$. For each prime $p\ge7$,
almost every integer of any fixed $p$-adic valuation at least $p+3$ occurs.
\end{abstract}
\maketitle

\section{Introduction}\label{sec:intro}

For every prime $p\ge7$, we determine the least exponent $e$ such that
every multiple of $p^e$ is an integer group determinant of $C_p\times C_p$.
Mossinghoff and Pinner~\cite[Theorems~2.3 and~4.2]{MossinghoffPinner2026}
showed that every multiple of $p^{2p+2}$ occurs for odd primes $p$ and
that every multiple of $p^{p+5}$ occurs for $p=7,11,13$. We obtain the
exact exponent from the residues modulo $p$ of the global units of
$\Z[\zeta_p]$. When Vandiver's conjecture holds at $p$, the formula uses
only the Bernoulli irregular pairs of $p$. In particular, the exponent
is $p+5$ for every regular prime $p\ge7$, which proves that Mossinghoff
and Pinner's bound is sharp at $p=11$ and $13$. At $p=103$ their
general bound is $e(103)\le208$, while the exact exponent is $116$.
This question is part of
the Taussky--Todd problem~\cite{TausskyTodd1977} of determining which
integers occur as determinants of integral group matrices. Laquer and
Newman~\cite{Laquer1980,Newman1980} studied the cyclic case, and Mahoney
and Newman~\cite{MahoneyNewman1980} treated abelian groups.

For a finite group $G$ and $v=(v_g)_{g\in G}\in\Z^G$, write
\[
 \Theta(v)=\det(v_{gh^{-1}})_{g,h\in G},\qquad
 S(G)=\{\Theta(v):v\in\Z^G\}.
\]
When $G=C_p\times C_p$ with $p$ odd, Odoni's
criterion gives the values prime to $p$ as
\[
 \mathcal U_p=\{n\in\Z:n^{p-1}\equiv1\pmod{p^2}\},
\]
see~\cite{Odoni1983,DeSilvaPinner2014}
and~\cite[Chapter~2, Theorem~12]{Trafford1992}. The positive
$p$-adic valuations that occur are exactly the integers at least
$p+3$~\cite[Chapter~2, Lemma~13]{Trafford1992},
\cite[Theorem~2.1]{MossinghoffPinnerHeisenberg2024}. Let $e(p)$ be the least $e\ge1$ such that
$p^e\Z\subseteq S(C_p\times C_p)$. Determining this exponent still
leaves open how small a positive $p$-divisible determinant can be at a
lower valuation. At $p=103$, for instance, valuation $106$ occurs,
but we shall show that the least positive $103$-divisible determinant
is $103^{116}$.

\subsection*{The least exponent}

Fix a prime $p\ge7$ and put $K=\Q(\zeta_p)$ and $R=\Z[\zeta_p]$. The
exponent depends on which even logarithmic coefficients vanish on the
units of $R$. For $u\in R^\times$, let $c_j(u)$ be the coefficient of $T^j$ in
$\log(u(T)/u(0))$, where $u(T)$ is the image of $u$ under
$R/pR\cong\F[T]/(T^{p-1})$, $\zeta_p\mapsto\exp T$.
Write
\[
 \Irru=\{i:1\le i\le(p-3)/2,\ c_{2i}(u)=0
                 \text{ for every }u\in R^\times\}.
\]
Define $\Irr$ to be the set of indices $1\le i\le(p-3)/2$ for which
$p$ divides the numerator of the Bernoulli number $B_{2i}$. Each such $(p,2i)$ is an
\emph{irregular pair}, and by Kummer's criterion $p$ is regular exactly
when $\Irr$ is empty~\cite[Theorems~5.16 and~10.9]{Washington1997}.
The unit calculation in Section~\ref{sec:units} gives $\Irru\subseteq\Irr$.

For a subset $\mathcal I\subseteq\{1,\ldots,(p-3)/2\}$, call $r\ge1$
\emph{admissible for $\mathcal I$} if
\[
 (2r+1)i>r(p-1)\qquad\text{for every }i\in\mathcal I\text{ with }r\mid i,
\]
and write $r(\mathcal I)$ for the least admissible $r$. Every
$r>(p-3)/2$ is admissible. Put $r_p=r(\Irr)$ and $r^u=r(\Irru)$.
We use $h_p^+$ for the class number of $\Q(\zeta_p+\zeta_p^{-1})$ and
$\ord_p(\ell)$ for the multiplicative order of $\ell$ modulo $p$.

\begin{theoremA}
Let $p\ge7$ be prime.
\begin{enumerate}
\item[\textup{(a)}] Every multiple of $p^{p+3+2r^u}$ is an integer group
 determinant of $C_p\times C_p$. In particular,
 $p^{p+3+2r_p}\Z\subseteq S(C_p\times C_p)$.
\item[\textup{(b)}] Let $m$ be an integer prime to $p$ such that
$\ell^{\ord_p(\ell)}\nmid m$ for every prime $\ell\mid m$, and let $N$
be an integer with $p+3\le N\le2p-2$ and $N<p+3+2r^u$. Then
$\pm p^Nm\notin S(C_p\times C_p)$.
\item[\textup{(c)}] We have
\begin{gather}
 e(p)=p+3+2r^u,\label{eq:exact-exponent}\\
 e(p)\le\left\lfloor\frac{7(p+1)}4\right\rfloor.\label{eq:uniform-bound}
\end{gather}
\item[\textup{(d)}] If $p\nmid h_p^+$, then $\Irru=\Irr$ and $r^u=r_p$;
hence $e(p)=p+3+2r_p$.
\end{enumerate}
\end{theoremA}

Part~(b) applies whenever $1\le |m|<p$. More generally its hypothesis
on $m$ says that no norm of a prime ideal of $R$ divides $m$
\cite[Theorem~2.13]{Washington1997}. For these cofactors, the lower
valuation bound, parts~(a) and~(b), and the bound $e(p)\le2p-2$ from
part~(c) and Lemma~\ref{lem:admissible-range} give
$p^Nm\in S(C_p\times C_p)$ exactly for $N\ge e(p)$ among positive $N$. Taking
$m=1$ identifies $p^{e(p)}$ as the least positive power of $p$ in
$S(C_p\times C_p)$.

Hart, Harvey and Ong verified Vandiver's conjecture, $p\nmid h_p^+$,
for primes below $2^{31}$~\cite{HartHarveyOng2017}. Part~(d) therefore
gives $e(p)=p+3+2r_p$ throughout $7\le p<2^{31}$.

\begin{corollary}\label{cor:intro-p5}
Let $p\ge7$ be prime with $p\nmid h_p^+$. Then
$p^{p+5}\Z\subseteq S(C_p\times C_p)$ if and only if every irregular
pair $(p,2i)$ satisfies $3i>p-1$.
\end{corollary}

\begin{corollary}\label{cor:least}
Let $p\ge7$ be prime. If $\Theta=\pm p^{p+3}m\in S(C_p\times C_p)$
with $p\nmid m$, then $|m|\ge(p+1)^{p-6}$.
If $p\ge11$, every nonzero $p$-divisible $\Theta\in S(C_p\times C_p)$
satisfies $|\Theta|\ge p^{p+5}$. Equality occurs if and only if $1$ is
admissible for $\Irru$; in particular it occurs whenever every irregular
pair $(p,2i)$ satisfies $3i>p-1$, and, if $p\nmid h_p^+$, only then.
\end{corollary}

The next corollary gives a sufficient condition for $p^{e(p)}$ itself
to be the least positive $p$-divisible value. It is automatic when
$r^u=1$.

\begin{corollary}[Least positive divisible determinant]\label{cor:minimum}
Let $p\ge11$ be prime and put $r=r^u$. Suppose that for each integer
$1\le a<r$ there is an $i\in\Irru$ with
\begin{equation}\label{eq:minimum-test}
 a\mid i,\qquad (2a+1)i/a\le p-2(r-a)-1.
\end{equation}
Then the least positive $p$-divisible determinant of $C_p\times C_p$
is $p^{e(p)}=p^{p+3+2r}$.
\end{corollary}

The criterion holds for each of the $664\,575$ primes $11\le p<10^7$
(Section~\ref{sec:examples}). The first prime with $e(p)>p+5$ is
$103$, where it gives the minimum $103^{116}$
(Corollary~\ref{cor:p103}). At $p=7$ the minimum is $43\cdot7^{10}<7^{e(7)}$.
These minima concern the Lind--Lehmer problem restricted to $p$-divisible
values. The original problem asks for the least absolute value greater
than one among the determinants for a given
group~\cite{Lind2005,DeSilvaPinner2014}.

Each direction in $\mathbf P^1(\F)$ represents one Galois orbit of
nontrivial characters. Under the hypotheses of Theorem~A(b), the
character values of a proposed determinant are unit multiples of
powers of $1-\zeta_p$. In suitable coordinates their valuations at
$\mathfrak p=(1-\zeta_p)R$ are one in the directions $b\in\F$ and
$e_\infty\ge2$ in the remaining direction. The conditions for coming
from a single integer vector bound the degrees in $b$ of the
coefficients of the expansions in $T$. They force $e_\infty$ to be odd,
say $2r+1$. For each $i\in\Irru$ divisible by $r$, the logarithmic
coefficient at $T^{2i}$ must vanish at every $b\in\F$. If $(2r+1)i/r$
were at most $p-1$, it would be the exact degree of the coefficient as
a polynomial in $b$, which has too few roots. For a general cofactor
the coefficient still vanishes wherever the cofactor ideal is the unit
ideal. Counting roots bounds the number of proper ideals from below,
and the product of their norms divides the cofactor. This bound gives Corollary~\ref{cor:minimum}.

For the converse, when $r$ is admissible for $\Irru$, choose a polynomial
$P$ so that global units give the expansions $T\,P(b^{2r+1}T^{2r})$
and $T^{2r+1}$ in the finite and remaining directions, respectively.
Circular units suffice if $r$ is admissible for $\Irr$.

\subsection*{Regular primes and prime cofactors}

Let $p\ge7$ be regular. Then $r_p=1$, so every multiple of $p^{p+5}$ is
a determinant by Theorem~A(a), and it remains to describe the values of
$p$-adic valuation $p+3$ and $p+4$. In Section~\ref{sec:assignments} we define two sets $\mathcal C_{2,p}$
and $\mathcal C_{23,p}$ of positive integers prime to $p$. An integer
$m$ lies in $\mathcal C_{2,p}$ if the prime ideals above its prime
factors can be assigned to the $p+1$ Galois orbits of nontrivial
characters so that each orbit receives a principal ideal and the assignment
satisfies a condition $\Ltwo$: one of at most $\gcd(3,p-1)$ affine
systems over $\F$ is consistent. For $\mathcal C_{23,p}$ the assignment
may instead satisfy a second condition $\Lthree$: one of $p-1$ affine
systems, depending on $m\bmod p$, is consistent. The subscripts $2$ and $3$
refer to the valuation of the exceptional character value.

\begin{theoremB}
For every regular prime $p\ge7$,
\begin{equation}\label{eq:intro-whole-set}
 S(C_p\times C_p)=\mathcal U_p
 \cup\{\pm p^{p+3}m:m\in\mathcal C_{2,p}\}
 \cup\{\pm p^{p+4}m:m\in\mathcal C_{23,p}\}
 \cup p^{p+5}\Z.
\end{equation}
If $\Cl(R)$ is known and generators of principal ideals can be computed,
then membership in $\mathcal C_{2,p}$ and $\mathcal C_{23,p}$ is decided
by a finite computation, which also produces a vector realizing each
corresponding determinant.
\end{theoremB}

By Corollary~\ref{cor:prime-exclusion}, $\pm p^{p+3}\ell\notin S(C_p\times C_p)$
for every prime $p\ge11$ and every rational prime $\ell\ne p$. Put $h_K=\lvert\Cl(R)\rvert$. For a principal
prime ideal $\mathfrak q=yR$ with $\mathfrak q\nmid p$, let
$c_j(\mathfrak q)$ be the coefficient of $T^j$ in $\log(y(T)/y(0))$, where $y(T)$
is the residue of $y$ modulo $p$ as above; for odd
$j\ge3$ it does not depend on the generator $y$ (Section~\ref{sec:units}).

\begin{theoremC}
Let $p\ge11$ be a regular prime and let $\ell\ne p$ be a rational prime.
\begin{enumerate}
\item[\textup{(a)}] $\pm p^{p+4}\ell\in S(C_p\times C_p)$ if and only if
$\ell\equiv1\pmod p$ and one \textup(equivalently, every\textup) prime
ideal $\mathfrak q$ above $\ell$ is principal with
\begin{equation}\label{eq:prime-criterion}
 c_3(\mathfrak q)=c_5(\mathfrak q)=\cdots=c_{p-6}(\mathfrak q)=0
 \quad\text{and}\quad
 (c_{p-4}(\mathfrak q),c_{p-2}(\mathfrak q))\ne(0,0).
\end{equation}
\item[\textup{(b)}] The primes $\ell$ with $\pm p^{p+4}\ell\in
S(C_p\times C_p)$ have natural density, and Dirichlet density, equal to
\begin{equation}\label{eq:prime-density}
 \frac{p+1}{h_Kp^{(p-3)/2}}
\end{equation}
in the set of all rational primes.
\item[\textup{(c)}] Every such prime satisfies $\ell>p^{p-4}/2^{p-1}$.
\end{enumerate}
\end{theoremC}

For regular $p\ge11$, Proposition~\ref{prop:noncongruence} shows that,
up to a set of density zero, the primes in part~(b) are not the primes
in a union of residue classes.

\subsection*{Almost all cofactors}

A set $A\subseteq\Z$
has \emph{natural density} $\delta$ if
$\#\{n\in A:|n|\le x\}=2\delta x+o(x)$ as $x\to\infty$.

\begin{theoremD}
Let $p$ be an odd prime.
\begin{enumerate}
\item[\textup{(a)}] If $p\ge7$ and $N\ge p+3$, then the integers $m$
prime to $p$ with $p^Nm\notin S(C_p\times C_p)$ form a set of natural
density zero.
\item[\textup{(b)}] If moreover $N<e(p)$, then the number of positive
integers $m\le x$ prime to $p$ with $p^Nm\notin S(C_p\times C_p)$ is
$\gg x/\log x$, and whether $p^Nm$ lies in $S(C_p\times C_p)$, for $m$
prime to $p$, does not depend only on the residue of $m$ modulo any
fixed positive integer.
\item[\textup{(c)}] The set $S(C_p\times C_p)$ has natural density
$(p-1)/p^2+p^{-(p+3)}$.
\end{enumerate}
\end{theoremD}

\subsection*{Methods and earlier work}

Complete descriptions of $S(G)$
are known for all groups of order at most $19$ and several infinite
families; see~\cite{MossinghoffPinner2026} for an overview.
For $C_p\times C_p$, the cases $p=3$ and $5$ are determined
in~\cite[Theorem~6.1]{PinnerSmyth2020} and~\cite{Panraksa2026Order25}.
Their least positive $p$-divisible values are $p^{p+3}$, and every
multiple of this number occurs. At $p=7$, the least positive
$p$-divisible value is $43\cdot7^{10}$, $e(7)=12$, and the values of
$7$-adic valuation $10$ and $11$ are described by invariants of prime
ideals~\cite[Theorem~1.1 and Corollary~1.2]{Panraksa2026Order49}.
That classification uses integral reconstruction and the image of the
global units modulo $7$. Together with the Chebotarev density theorem and
Lemma~\ref{lem:many-factors}, its criteria~\cite[Theorem~1.1]{Panraksa2026Order49}
also give Theorem~D at $p=7$, and there the last assertion of Theorem~D(b)
is equivalent to~\cite[Corollary~1.3]{Panraksa2026Order49}.

A vector $v\in\Z^G$ is determined by its projections to the $p+1$
quotients of $G$ by subgroups of order $p$, and a family of integer
vectors on these quotients comes from an integer vector exactly when
certain congruences modulo $p$ hold. This is the pull-back description
of $\Z[G]$ due to Hoechsmann~\cite{Hoechsmann2002}. We show that the
congruences hold exactly when the polynomial moments of the projections
modulo $p$ satisfy the degree conditions of Lemma~\ref{lem:moments}.
Trafford used such moments to bound the valuation~\cite{Trafford1992}.

The residues of the global units of $\Z[\zeta_p]$ modulo $p$ are
classical when $p\nmid h_p^+$: they are determined by Kummer's
logarithmic derivatives of the circular units and the index formula
$[E^+:C^+]=h_p^+$ for the real circular units $C^+$ in the real units
$E^+$~\cite[Theorem~8.2]{Washington1997}. For regular $p$ every unit of the integral group ring
of an elementary abelian $p$-group is a product of units coming from
cyclic subgroups~\cite{HoechsmannSehgal1986}, whereas for irregular $p$
and rank at least two there is a shortfall described by a finite
abelian $p$-group~\cite{Hoechsmann2002}.

At the valuations $p+3$ and $p+4$ the conditions on all orbits reduce
to affine systems over $\F$: once the cofactor ideals fix the odd
logarithmic coordinates, the remaining coefficients depend affinely on
the free even coefficients. Arithmetic restrictions on cofactors at the
lowest valuations also occur for the cyclic groups of orders $25$ and
$27$~\cite{Newman1980b,MossinghoffPinnerCirculants2023} and for
$C_3^3$~\cite[Theorem~5.1]{MossinghoffPinner2026}. The use of norms of
ideals in ray classes for circulant determinants goes back to
Trafford, who applied results of Odoni~\cite{Odoni1989} and showed that
at each possible valuation almost all cofactors of circulant determinants
occur~\cite[Chapter~3, Theorem~31]{Trafford1992}. Theorem~D is the
analogue for $C_p\times C_p$. Hoelscher~\cite[Lemma~3.7]{Hoelscher2010} determined the ray
class group of $\Q(\zeta_p)$ modulo $p$ for regular $p$, and Theorem~C
identifies the ray classes of the prime ideals above the prime
cofactors.

\section{Character values and integral reconstruction}\label{sec:characters}

The character values factor the determinant into the augmentation
and $p+1$ norms. To use those factors to construct determinants, we
also need to recover an integer vector from its character values.
In Sections~\ref{sec:characters}--\ref{sec:density}, $p\ge5$ is prime
and $G=C_p\times C_p$ is identified with $\F^2$. Put
$\zeta=\exp(2\pi\mathrm i/p)$, $K=\Q(\zeta)$, $R=\Z[\zeta]$,
$\varpi=1-\zeta$ and $\p=\varpi R$, and write $\Nm$ for the norm from $K$
to $\Q$. For $v=(v_{ij})_{(i,j)\in G}\in\Z^G$ let
$s=\sum_{i,j}v_{ij}$ be its augmentation, and define
\[
 f_b=\sum_{i,j}v_{ij}\zeta^{i+bj}\quad(b\in\F),\qquad
 f_\infty=\sum_{i,j}v_{ij}\zeta^j.
\]
Every nontrivial character of $G$ is a Galois conjugate of exactly one
of the $p+1$ characters $(i,j)\mapsto\zeta^{i+bj}$ ($b\in\F$) and
$(i,j)\mapsto\zeta^j$, whose values at $v$ are $f_b$ and $f_\infty$.
The nontrivial characters thus fall into $p+1$ Galois orbits of size
$p-1$, indexed by the directions
$\delta\in\mathbf P^1(\F)=\F\cup\{\infty\}$, and Fourier diagonalization
gives
\begin{equation}\label{eq:factor}
 \Theta(v)=s\prod_{\delta\in\mathbf P^1(\F)}\Nm(f_\delta).
\end{equation}
Every nonzero norm in~\eqref{eq:factor} is positive, since the complex
embeddings of $K$ occur in conjugate pairs. As $p$ is totally ramified
in $K$, we have $(p)=\p^{p-1}$ and $\Nm(\p)=p$;
see~\cite[Lemma~1.4]{Washington1997}. We normalize the valuation $v_\p$
by $v_\p(\varpi)=1$. If $\Theta(v)\ne0$, then
\begin{equation}\label{eq:valuation}
 v_p(\Theta(v))=v_p(s)+\sum_{\delta\in\mathbf P^1(\F)}v_\p(f_\delta).
\end{equation}
Composing $v$ with a linear automorphism of $G$ permutes the rows and
the columns of the group matrix in the same way, so it preserves
$\Theta$. It also permutes the $p+1$ directions and replaces the
character values by Galois conjugates. Multiplying $v$ by a group element
(a \emph{translation}) multiplies each $f_\delta$ by a root of unity and
preserves $\Theta$, because the product of all character values of a
group element of odd order is one. Finally, $\Theta(-v)=-\Theta(v)$
because $|G|=p^2$ is odd, and hence $S(C_p\times C_p)=-S(C_p\times C_p)$.

Let $L$ be the linear form
\[
 L(\alpha,\beta)=\sum_{i,j}v_{ij}(\alpha i+\beta j)\bmod p
 \qquad((\alpha,\beta)\in\F^2).
\]

\begin{lemma}\label{lem:valuations}
Let $v\in\Z^G$ with $\Theta(v)\ne0$ and $p\mid\Theta(v)$. Then $p\mid s$
and every $f_\delta$ lies in $\p$. If $L=0$, then
$v_p(\Theta(v))\ge2p+3$; in particular, $v_p(\Theta(v))<2p+3$ implies
$L\ne0$. If $L\ne0$, then after a linear change of
coordinates $v_\p(f_b)=1$ for every $b\in\F$, the exceptional valuation
$e_\infty=v_\p(f_\infty)$ is at least two, and
\begin{equation}\label{eq:profile-valuation}
 v_p(\Theta(v))=v_p(s)+p+e_\infty.
\end{equation}
In particular, a determinant of exact $p$-adic valuation $p+3$ has
$(v_p(s),e_\infty)=(1,2)$, and one of exact $p$-adic valuation $p+4$ has
$(v_p(s),e_\infty)=(2,2)$ or $(1,3)$.
\end{lemma}

\begin{proof}
Every character value is congruent to $s$ modulo $\p$. If $p\nmid s$,
then all character values and their norms are prime to $p$, and so is
$\Theta(v)$, contrary to the hypothesis. Hence $p\mid s$ and every
$f_\delta$ lies in $\p$. Since
$\zeta^k\equiv1-k\varpi\pmod{\p^2}$ and $s\in\p^{p-1}\subseteq\p^2$, the
value of the character $(i,j)\mapsto\zeta^{\alpha i+\beta j}$ is
congruent to $-\varpi L(\alpha,\beta)$ modulo $\p^2$. If $L=0$, all
$p+1$ values lie in $\p^2$, and by~\eqref{eq:valuation} the valuation is
at least $1+2(p+1)$. Otherwise the kernel of $L$ is a single line, and a
linear change of coordinates makes it the line $\alpha=0$, which
corresponds to $f_\infty$. Then $L(1,b)\ne0$ for every $b\in\F$, so
$v_\p(f_b)=1$, while $v_\p(f_\infty)\ge2$, and~\eqref{eq:valuation}
becomes~\eqref{eq:profile-valuation}. The cases of $p$-adic valuation $p+3$ and
$p+4$ follow because $v_p(s)\ge1$ and $e_\infty\ge2$.
\end{proof}

In these coordinates, $\infty$ is the exceptional direction.

For $b\in\F$ the \emph{projection} of $v$ in direction $b$ is the
function $\pi_b\colon\F\to\Z$ given by $\pi_b(t)=\sum_{i+bj=t}v_{ij}$,
and the projection in direction $\infty$ is
$\pi_\infty(t)=\sum_{j=t}v_{ij}$. In other words, $\pi_\delta$ is the
image of $v$ under the map $\Z[G]\to\Z[\F]\cong\Z[C_p]$ induced by the
homomorphism $(i,j)\mapsto i+bj$ when $\delta=b$, or $(i,j)\mapsto j$
when $\delta=\infty$. Each projection
has sum $s$, and its cyclotomic evaluation $\sum_t\pi_\delta(t)\zeta^t$
is $f_\delta$. Conversely, an element $f\in R$ with $f\equiv s\pmod\p$
is the evaluation of exactly one integer vector on $\F$ with sum $s$:
two representatives of $f$ by integer vectors of length $p$ differ by a
constant vector, and their sums differ by a multiple of $p$.

The next lemma, including~\eqref{eq:integer-recovery}, is the rank-two
case of Hoechsmann's pull-back description of the integral group ring of
an elementary abelian $p$-group~\cite[Lemma~4 and~(2.6)]{Hoechsmann2002}.

\begin{lemma}[Integral reconstruction]\label{lem:reconstruct}
Let $\pi_\delta\colon\F\to\Z$, $\delta\in\mathbf P^1(\F)$, be integer
vectors with common sum $s$. They are the projections of some
$v\in\Z^G$ if and only if
\begin{equation}\label{eq:reconstruction-congruence}
 \pi_\infty(j)+\sum_{b\in\F}\pi_b(i+bj)\equiv s\pmod p
 \qquad\text{for every }(i,j)\in\F^2.
\end{equation}
In that case $v$ is unique, its augmentation is $s$, and
\begin{equation}\label{eq:integer-recovery}
 v_{ij}=\frac{\pi_\infty(j)+\sum_{b\in\F}\pi_b(i+bj)-s}{p}.
\end{equation}
\end{lemma}

\begin{proof}
For the projections of a given $v$, the numerator
of~\eqref{eq:integer-recovery} before subtracting $s$ counts the cell
$(i,j)$ once on each of the $p+1$ lines through it and every other cell
exactly once. It therefore equals $pv_{ij}+s$, which proves necessity
and uniqueness. Conversely, the congruences
make~\eqref{eq:integer-recovery} integral. Sum this formula over the
line in direction $\delta$ indexed by $t$. The term $\pi_\delta$
contributes $p\pi_\delta(t)$, each of the other $p$ directions
contributes $s$, and the $p$ subtractions of $s$ cancel these
contributions. The reconstructed vector thus has the prescribed
projections, and summing over all $p^2$ cells gives the augmentation
$((p+1)ps-p^2s)/p=s$.
\end{proof}

The congruences~\eqref{eq:reconstruction-congruence} involve the
projections only modulo $p$, and these residues are determined by the
character values modulo $p\p=\p^p$. For $1\le e<p$ and $z\in R$, let $\varrho_e(z)$
denote the integer vector on $\F$ with sum $p$ and evaluation
$\varpi^ez$. Evaluation identifies the lattice of integer vectors of sum
zero with $\p$, and a vector of sum zero is divisible by $p$ exactly when
its evaluation lies in $p\p=\p^p$. Consequently, writing $\overline c$
for the reduction modulo $p$ of the entries of an integer vector $c$, we
have
\begin{equation}\label{eq:residues}
 \overline{\varrho_e(z)}=\overline{\varrho_e(z')}
 \quad\Longleftrightarrow\quad z\equiv z'\pmod{\p^{p-e}}.
\end{equation}

We use the logarithmic variable $T=\log\zeta$. Since
$\Phi_p(X)\equiv(X-1)^{p-1}$ modulo $p$ and $\exp T-1$ is $T$ times a
unit, the substitution $\zeta\mapsto\exp T$, with the exponential
truncated below degree $p$, identifies $R/pR$ with $\F[T]/(T^{p-1})$. All logarithms and
exponentials below are truncated in the same way, and the denominators
that occur are prime to $p$. For $1\le e<p$ and $z\in R$ with residue
$z(T)\in\F[T]/(T^{p-1})$ we put
\begin{equation}\label{eq:H}
 H(T)=(1-\exp T)^ez(T)\pmod{T^p},\qquad
 \exp(tT)=\sum_{n=0}^{p-1}\frac{t^nT^n}{n!}.
\end{equation}
Since $(1-\exp T)^e$ has $T$-adic order $e$, the series $H$ depends only
on $z$ modulo $\p^{p-e}$, that is, on $\varpi^ez$ modulo $\p^p$. We call $H$ the \emph{expansion}
of $\varpi^ez$; it is the image of $\varpi^ez$ under the $\F$-module
isomorphism
\[
 \p/\p^p\longrightarrow T\F[T]/(T^p),\qquad
 \varpi^j\longmapsto(1-\exp T)^j\quad(1\le j\le p-1).
\]

\begin{lemma}\label{lem:expansions}
Let $1\le e<p$ and $z\in R$, and let $H$ be the expansion of $\varpi^ez$.
Then
\begin{equation}\label{eq:expansion-sum}
 \sum_{t\in\F}\overline{\varrho_e(z)}(t)\exp(tT)=T^{p-1}+H(T).
\end{equation}
If $\overline{\varrho_e(z)}(t)=\sum_{d=0}^{p-1}\gamma_dt^d$ as a polynomial
function on $\F$ and $h_n=[T^n]H$ denotes the coefficient of $T^n$ in
$H$, then
\begin{equation}\label{eq:gamma}
 \gamma_{p-1}=0,\qquad
 \gamma_{p-1-n}=-n!\,h_n\ (1\le n\le p-2),\qquad
 \gamma_0=1+h_{p-1}.
\end{equation}
\end{lemma}

\begin{proof}
Represent $\varpi^ez$ modulo $p\p$ by a polynomial in $\varpi$ without
constant term. Expanding this polynomial in powers of $\zeta$ gives a
vector $(a_t)_{t\in\F}$ of sum zero with $\sum_ta_t\exp(tT)=H(T)$.
Adding the all-ones vector raises the sum to $p$ and adds
$\sum_t\exp(tT)=T^{p-1}$, because the power sums $\sum_tt^n$ vanish in
$\F$ for $0\le n\le p-2$ and $(p-1)!=-1$ in $\F$. This
proves~\eqref{eq:expansion-sum}. Comparing constant coefficients gives
$\gamma_{p-1}=0$. For $1\le n\le p-2$, the coefficient of $T^n$ on the
left of~\eqref{eq:expansion-sum} is $\sum_{t,d}\gamma_dt^{d+n}/n!$. Only the
power sum with $d+n=p-1$ is nonzero, and the coefficient equals
$-\gamma_{p-1-n}/n!$. At $n=p-1$ the remaining term is $-\gamma_0$
divided by $(p-1)!=-1$. These are the formulas~\eqref{eq:gamma}.
\end{proof}

For a direction $\delta$ let $H_\delta$ be the expansion of $f_\delta$,
defined by~\eqref{eq:H} with $e=v_\p(f_\delta)$ when $v_\p(f_\delta)<p$;
if $v_\p(f_\delta)\ge p$, then $f_\delta\in p\p$ and we put $H_\delta=0$.
We write $h_n(b)=[T^n]H_b$ and $h_{\infty,n}=[T^n]H_\infty$. When
$p\mid s$, the projection $\pi_b$ differs from the integer vector with
sum $p$ and evaluation $f_b$ by the constant vector with entries
$s/p-1$. By Lemma~\ref{lem:expansions},
\begin{gather*}
 n!\,h_n(b)\equiv\sum_{i,j}v_{ij}(i+bj)^n\pmod p\qquad(1\le n\le p-2),\\
 h_{p-1}(b)\equiv\pi_b(0)-s/p\pmod p,
\end{gather*}
and similarly for $\delta=\infty$; if $v_\p(f_b)\ge p$, all these
quantities vanish. Up to the factors $n!$, the coefficients
$h_1(b),\dots,h_{p-2}(b)$ are thus moments of $v$ modulo $p$, but
$h_{p-1}(b)$ is not. Functions of $b\in\F$ are represented by
their unique polynomials of degree at most $p-1$, and products of such
functions are reduced modulo $b^p-b$.

\begin{lemma}\label{lem:moments}
Suppose that $p\mid s$, and put $\tau=s/p\bmod p$. Integer vectors
$\pi_\delta\colon\F\to\Z$ of common sum $s$ with evaluations
$f_\delta\in\p$ satisfy the reconstruction
congruences~\eqref{eq:reconstruction-congruence} if and only if
\begin{equation}\label{eq:degree}
 \deg h_n\le n\quad\text{and}\quad h_{\infty,n}=[b^n]h_n
 \qquad(1\le n\le p-2),
\end{equation}
and
\begin{equation}\label{eq:constant}
 \sum_{b\in\F}h_{p-1}(b)+h_{\infty,p-1}=-\tau.
\end{equation}
\end{lemma}

\begin{proof}
First suppose that every $\pi_\delta$ has sum $p$, and
use~\eqref{eq:gamma}. Regard the reconstruction sum
in~\eqref{eq:reconstruction-congruence} as a polynomial function of
$(X,Y)=(i,j)$. For $1\le d\le p-2$, its homogeneous part of degree $d$ is
\[
 \sum_{b\in\F}\gamma_d(b)(X+bY)^d+\gamma_{\infty,d}Y^d.
\]
Its coefficients vanish exactly when $\sum_bb^k\gamma_d(b)=0$ for
$0\le k<d$ and $\gamma_{\infty,d}=-\sum_bb^d\gamma_d(b)$, because the
binomial coefficients involved are nonzero modulo $p$. Taking power sums
for $k=0,1,\dots,d-1$ in turn shows that the first conditions say that
$\gamma_d(b)$ has degree at most $p-1-d$. The last condition identifies
$\gamma_{\infty,d}$ with the coefficient of $b^{p-1-d}$. With $n=p-1-d$
and after cancelling the common factor $-n!$, these are the
conditions~\eqref{eq:degree}. No projection contributes a term of degree
$p-1$, since $\gamma_{p-1}=0$. Polynomial functions on $\F^2$ of degree at most
$p-1$ in each variable are determined by their values, so this accounts
for every nonconstant term.

Changing the sum of every projection from $p$ to $s$ adds $s/p-1$ to
each entry. At sum $p$ the $p+1$ constant terms contribute
$1+\sum_bh_{p-1}(b)+h_{\infty,p-1}$ modulo $p$, and the adjustments
contribute $(p+1)(s/p-1)\equiv\tau-1$. Since $s\equiv0\pmod p$, the
remaining congruence is~\eqref{eq:constant}.
\end{proof}

We call~\eqref{eq:degree} the \emph{degree conditions}
and~\eqref{eq:constant} the \emph{augmentation condition}.

\begin{lemma}[Lifting]\label{lem:lifting}
Let $s$ be a nonzero integer divisible by $p$, and let $f_\delta\in\p$,
$\delta\in\mathbf P^1(\F)$, be elements whose expansions $H_\delta$
satisfy~\eqref{eq:degree} and~\eqref{eq:constant}
with $\tau=s/p\bmod p$. Then there is a unique $v\in\Z^G$ with augmentation
$s$ and character values $f_\delta$, and
$\Theta(v)=s\prod_\delta\Nm(f_\delta)$.
\end{lemma}

\begin{proof}
Let $f$ be one of the $f_\delta$. Write $f=\sum_{t=0}^{p-2}a_t\zeta^t$
with $a_t\in\Z$, and define
\begin{equation}\label{eq:lift-vector}
 k_f=\frac{s-\sum_{t=0}^{p-2}a_t}{p},\qquad
 \pi_f=(a_0+k_f,\ldots,a_{p-2}+k_f,k_f).
\end{equation}
The number $k_f$ is an integer, because $f\in\p$ implies
$\sum_ta_t\equiv0\pmod p$ and $p\mid s$. The vector $\pi_f$ has sum $s$
and evaluation $f$. Lemma~\ref{lem:moments} now shows that the vectors $\pi_{f_\delta}$
satisfy the reconstruction
congruences~\eqref{eq:reconstruction-congruence}, and by
Lemma~\ref{lem:reconstruct} there is $v\in\Z^G$ with these projections and
augmentation $s$. Its character values are the evaluations $f_\delta$.
Any $v'\in\Z^G$ with augmentation $s$ and character values $f_\delta$
has projections of sum $s$ and evaluations $f_\delta$. These are the
vectors $\pi_{f_\delta}$, because an integer vector on $\F$ with sum
$s$ is determined by its evaluation, and $v'=v$ by
Lemma~\ref{lem:reconstruct}. Finally,
$\Theta(v)=s\prod_\delta\Nm(f_\delta)$ by~\eqref{eq:factor}.
\end{proof}

A translation by $g=(g_1,g_2)\in G$ multiplies $f_b$ by $\zeta^{g_1+bg_2}$ and
$f_\infty$ by $\zeta^{g_2}$. It shifts the projections cyclically and
therefore multiplies $H_b$ by $\exp((g_1+bg_2)T)$ and $H_\infty$ by
$\exp(g_2T)$. In particular it adds $(g_1+bg_2)h_1(b)$ to $h_2(b)$ and leaves
$h_1$ unchanged.

\section{Units, Bernoulli numbers and coordinates of ideals}\label{sec:units}

A principal ideal determines its generator up to a global unit.
To see which coefficients can be changed by choosing a different
generator, write $y(T)\in\F[T]/(T^{p-1})$ for the residue of an element
$y\in R$ prime to $\p$. Its \emph{logarithmic coordinates} are the
normalized Kummer logarithmic derivatives
\begin{equation}\label{eq:log-coordinates}
 \log\frac{y(T)}{y(0)}=\sum_{j=1}^{p-2}c_j(y)T^j.
\end{equation}
We have $c_j(yy')=c_j(y)+c_j(y')$. For
$t\in\F^\times$ let $\sigma_t$ be the automorphism of $K$ with
$\sigma_t(\zeta)=\zeta^t$. Since $\sigma_t$ replaces $T$ by $tT$, we have
$c_j(\sigma_ty)=t^jc_j(y)$. The residue ring
$R/pR=\F[T]/(T^{p-1})$ has unit group
$\F^\times\times(1+T\F[T]/(T^{p-1}))$. The second factor is an
elementary abelian $p$-group, and the truncated logarithm maps it
isomorphically and Galois-equivariantly onto the additive group
$T\F[T]/(T^{p-1})$. Let $\Ehat$ denote the image of the unit group
$R^\times$ in $(R/pR)^\times$.

The Bernoulli numbers enter through the series $E(T)=(\exp T-1)/T$. Put
$M=(p-3)/2$. Integrating the identity
\[
 \frac{\mathrm d}{\mathrm dT}\log E(T)=\frac12+\sum_{k\ge2}\frac{B_kT^{k-1}}{k!}
\]
and reducing modulo $p$, we obtain
\begin{equation}\label{eq:uniformizer}
 \log E(T)=\frac T2+\sum_{i=1}^{M}\kappa_iT^{2i}\pmod{T^{p-1}},
 \qquad \kappa_i=\frac{B_{2i}}{2i\,(2i)!}\bmod p.
\end{equation}
The denominators are prime to $p$ in this range, since by the theorem of
von Staudt and Clausen $p$ divides the denominator of $B_k$ only when
$p-1\mid k$~\cite[Theorem~5.10]{Washington1997}. Hence $\Irr=\{i:1\le i\le M,\ \kappa_i=0\}$.

Part~(ii) below is Kummer's computation of the logarithmic derivatives
of the circular units. The odd-coefficient vanishing in part~(i)
and the description of the even coefficients for regular primes
correspond to the unit filtration in
\citet[Lemmas~3.2--3.3]{Hoelscher2010}.

\begin{lemma}\label{lem:units}
Let $p\ge5$ be prime.
\begin{enumerate}
\item[\textup{(i)}] Every $u\in R^\times$ satisfies $c_j(u)=0$ for each
odd $j$ with $3\le j\le p-2$.
\item[\textup{(ii)}] For $1\le a\le p-1$, the circular unit
$u_a=(1-\zeta^a)/(1-\zeta)$ satisfies
\begin{equation}\label{eq:unit-expansion}
 \log\frac{u_a(T)}a=\frac{a-1}{2}\,T+
 \sum_{i=1}^{M}\kappa_i(a^{2i}-1)T^{2i}.
\end{equation}
\item[\textup{(iii)}] The image in $(R/pR)^\times$ of the group $C$
generated by $\zeta$ and the $u_a$ is
\[
 \widehat C=\F^\times\cdot\exp\Bigl(\F T\oplus
 \bigoplus_{i\notin\Irr}\F T^{2i}\Bigr).
\]
\item[\textup{(iv)}] $\Irru\subseteq\Irr$ and
\[
 \Ehat=\F^\times\cdot\exp\Bigl(\F T\oplus
 \bigoplus_{i\notin\Irru}\F T^{2i}\Bigr).
\]
\item[\textup{(v)}] If $p\nmid h_p^+$, then $\Ehat=\widehat C$ and
$\Irru=\Irr$.
\end{enumerate}
In particular, if $p$ is regular, then
$\Ehat=\F^\times\cdot\exp\bigl(\F T\oplus\bigoplus_{i=1}^{M}\F T^{2i}\bigr)$; two
residues prime to $\p$ then differ by the residue of a global unit if and
only if their coordinates $c_3,c_5,\ldots,c_{p-2}$ agree.
\end{lemma}

\begin{proof}
(i) Every unit of $R$ is a power of $\zeta$ times a real unit;
see~\cite[Proposition~1.5]{Washington1997}. Complex conjugation replaces
$T$ by $-T$, so for a real unit $u$ the series $\log(u(T)/u(0))$ is even, and
$\zeta$ contributes only the linear term $T$.

(ii) We have $u_a(T)=(\exp(aT)-1)/(\exp T-1)=aE(aT)/E(T)$, and the
expansion~\eqref{eq:unit-expansion} follows by applying the
series~\eqref{eq:uniformizer} to $E(aT)$ and $E(T)$.

(iii) Since $u_a\equiv a\pmod\p$ and $1+T\F[T]/(T^{p-1})$ has exponent
$p$, the residue of $u_a^p$ is the constant $a$; hence
$\F^\times\subseteq\widehat C$. By~(ii), the logarithms of the residues
of $\zeta$ and $u_a$ span $\F T$ together with the vectors
$\sum_i\kappa_i(a^{2i}-1)T^{2i}$, and the logarithmic part of
$\widehat C$ is contained in the displayed space, the coordinates with
$i\in\Irr$ being zero. For the reverse inclusion consider the square
matrix $(a^{2i}-1)_{1\le i\le M,\ 2\le a\le M+1}$. It is invertible over
$\F$. If a linear combination $\sum_i\alpha_i(a^{2i}-1)$ of its rows
vanishes, then the polynomial $\sum_{i=1}^{M}\alpha_i(X^i-1)$, of degree
at most $M$, has the $M+1$ roots $1,2^2,\ldots,(M+1)^2$, which are
distinct because $M+1=(p-1)/2$; every $\alpha_i$ is then zero. Deleting
the rows with $i\in\Irr$ and multiplying each remaining row by
$\kappa_i\ne0$ preserves the independence of the rows. The units
$u_2,\ldots,u_{M+1}$ thus realize every vector of coordinates
$(c_{2i})_{i\notin\Irr}$, and a power of $\zeta$ adjusts the linear
coordinate.

(iv) By~(i), the logarithmic part of $\Ehat$ is contained in
$\F T\oplus\bigoplus_i\F T^{2i}$. Being a subgroup of an $\F$-vector
space, it is a subspace, and it is stable under the automorphisms
$\sigma_t$. Since the characters $t\mapsto t^j$ for $1\le j\le p-2$ are
distinct, the operator $\mathrm{pr}_j=(p-1)^{-1}\sum_tt^{-j}\sigma_t$
maps $\sum_kx_kT^k$ to $x_jT^j$ and preserves this subspace. Hence the
subspace is the direct sum of the lines $\F T^j$ for which some element
has a nonzero coordinate $x_j$. It contains $T$, and by the definition of
$\Irru$ some unit has a nonzero coordinate $c_{2i}$ exactly when
$i\notin\Irru$. Together with $\F^\times\subseteq\Ehat$ this gives the
displayed description. By~(iii), every $i\notin\Irr$ lies outside
$\Irru$.

(v) Let $C^+$ be the group generated by $-1$ and the real circular units
$\zeta^{(1-a)/2}u_a$ with $1<a<p/2$, where $(1-a)/2$ is read modulo $p$.
These generate the circular units of
$\Q(\zeta+\zeta^{-1})$~\cite[Lemma~8.1]{Washington1997}, and the index
$[E^+:C^+]$ of $C^+$ in the unit group $E^+$ of $\Q(\zeta+\zeta^{-1})$
equals $h_p^+$~\cite[Theorem~8.2]{Washington1997}. As recalled in~(i),
$R^\times=\langle\zeta\rangle E^+$. Since $C$ contains
$\langle\zeta\rangle C^+$ (note that $-1=\zeta u_{p-1}$), every
$u\in R^\times$ satisfies $u^{h_p^+}\in C$. The $p$-group
$1+T\F[T]/(T^{p-1})$ has exponent $p$, and as $p\nmid h_p^+$, raising to
the power $h_p^+$ is bijective on it. Hence the image of $u$ in this
$p$-group lies in the image of $C$. As $\widehat C$ also contains the
constants $\F^\times$, we obtain $\Ehat=\widehat C$, and comparing~(iii)
with~(iv) gives $\Irru=\Irr$.

If $p$ is regular, then $\Irr=\varnothing$, and the last assertion
follows from~(iv).
\end{proof}

The factor $\varpi^e$ in a character value changes only the linear and
even coordinates: $1-\exp T=-TE(T)$, and by~\eqref{eq:uniformizer} the
series $\log E(T)$ has no odd terms of degree at least three. More
precisely, let $f=\varpi^ez$ with $1\le e=v_\p(f)<p$ and $z$ prime
to $\p$, and let $H$ be its expansion. Since $H(T)=(-1)^eT^eE(T)^ez(T)$,
its leading coefficient is $h_e=(-1)^ez(0)\ne0$, and
$H/(h_eT^e)$ is defined modulo $T^{p-e}$. Taking logarithms gives
\begin{equation}\label{eq:relative-log}
 \log\frac{H(T)}{h_eT^e}=e\log E(T)+\log\frac{z(T)}{z(0)}
 \pmod{T^{p-e}},\qquad h_e=(-1)^ez(0)\ne0.
\end{equation}
By~\eqref{eq:relative-log}, for odd $j$ with $3\le j\le p-1-e$ the
logarithmic coordinate of degree $j$ of $H/(h_eT^e)$ is $c_j(z)$, and for
$2\le2i\le p-1-e$ its coordinate of degree $2i$ is $e\kappa_i+c_{2i}(z)$.
For a series $1+qT+g_2T^2+g_3T^3+\cdots$ with coefficients in $\F$, or in
the ring $\F[b]/(b^p-b)$ of functions of $b\in\F$, the logarithmic
coordinate of degree three is
\begin{equation}\label{eq:cubic-coordinate}
 [T^3]\log\bigl(1+qT+g_2T^2+g_3T^3+\cdots\bigr)=g_3-qg_2+\frac{q^3}{3}.
\end{equation}

\begin{lemma}\label{lem:fixed-ideal}
Let $p\ge5$ be regular, let $1\le e<p$, and let $y\in R$ be prime to
$\p$. As $u$ ranges over $R^\times$, the expansions of
$\varpi^eyu$ are exactly the series
\begin{equation}\label{eq:fixed-ideal}
 H(T)=\lambda T^e\exp\Biggl(
 A_1T+\sum_{\substack{2\le j\le p-1-e\\j\text{ even}}}A_jT^j+
 \sum_{\substack{3\le j\le p-1-e\\j\text{ odd}}}c_j(y)T^j
 \Biggr)\pmod{T^p},
\end{equation}
where $\lambda\in\F^\times$ and the $A_j\in\F$ are arbitrary.
\end{lemma}

\begin{proof}
Only the coordinates of $yu$ below degree $p-e$ affect its expansion.
By Lemma~\ref{lem:units}(i) and~\eqref{eq:relative-log}, every such
expansion has the form~\eqref{eq:fixed-ideal}. Conversely, given a series of the form~\eqref{eq:fixed-ideal}, choose a
residue $z(T)$ with $z(0)=(-1)^e\lambda$, with the coordinates of degree
less than $p-e$ that~\eqref{eq:relative-log} requires, and with the odd
coordinates of degree at least $p-e$ equal to those of $y$. By the last assertion of Lemma~\ref{lem:units}, it is the
residue of $yu$ for some $u\in R^\times$, so every such series occurs.
\end{proof}

\begin{lemma}\label{lem:ideals}
Suppose that $p$ is regular, and let $h_K=\lvert\Cl(R)\rvert$ be the class number
of $K$. For an integral ideal $I$ prime to $\p$, choose a generator $y$
of $I^{h_K}$ and set
\begin{equation}\label{eq:ideal-coordinates}
 c_j(I)=h_K^{-1}c_j(y)\in\F\qquad(j=3,5,\ldots,p-2).
\end{equation}
These coordinates are well defined, additive and Galois equivariant, and
they equal $c_j(z)$ when $I=zR$. For a prime ideal $\mathfrak q$ above a
rational prime $\ell\ne p$,
\begin{equation}\label{eq:frobenius}
 (\ell^j-1)\,c_j(\mathfrak q)=0\qquad(j=3,5,\ldots,p-2).
\end{equation}
\end{lemma}

\begin{proof}
The class number $h_K$ annihilates the class group and, by regularity,
is prime to $p$. Hence $I^{h_K}$ is principal, its generators are prime
to $\p$, and any two of them differ by a unit $u$. By
Lemma~\ref{lem:units}(i), $c_j(u)=0$ for $j=3,5,\ldots,p-2$. Multiplying
generators of $I^{h_K}$ and $J^{h_K}$ proves additivity, and applying
$\sigma_t$ to a generator proves equivariance. If $I=zR$, we may choose
$y=z^{h_K}$. Finally, the Frobenius automorphism $\sigma_\ell$ fixes
$\mathfrak q$, and equivariance gives
$\ell^jc_j(\mathfrak q)=c_j(\mathfrak q)$.
\end{proof}

\section{Proof of Theorem A}\label{sec:tail}

We separate the powers of $\p$ in the character values from the
ideals prime to $\p$. Throughout this section $p\ge7$. For a nonzero
determinant $\pm p^Nm$ with $p\nmid m$, call $m$ its \emph{cofactor}.
Write $f_\delta R=\p^{e_\delta}I_\delta$, where $I_\delta$ is the
integral ideal prime to $\p$.

\begin{lemma}\label{lem:unit-cofactors}
Let $v\in\Z^G$ with $\Theta(v)=\pm p^Nm$, where $1\le N\le2p+2$, $m$ is
prime to $p$, and $\ell^{\ord_p(\ell)}\nmid m$ for every prime $\ell\mid m$.
Then each of the character values $f_\delta$,
$\delta\in\mathbf P^1(\F)$, has the form
$f_\delta=\varpi^{e_\delta}z_\delta$ with $z_\delta\in R^\times$. After a
linear change of coordinates, $e_b=1$ for every $b\in\F$, the
exceptional valuation $e_\infty$ is at least two, and
$N=v_p(s)+p+e_\infty$.
\end{lemma}

\begin{proof}
Let $s'$ be the part of $s$ prime to $p$. By the
factorization~\eqref{eq:factor}, the part of $|\Theta(v)|$ prime to $p$
equals $|s'|\prod_\delta\Nm(I_\delta)$. If some $I_\delta$ were a proper
ideal, the norm of one of its prime factors would divide $m$, contrary
to the hypothesis. Every $I_\delta$ is therefore $R$, and the coordinates
are those of Lemma~\ref{lem:valuations}.
\end{proof}

\begin{lemma}\label{lem:normal-form}
Let $v\in\Z^G$ with $\Theta(v)\ne0$, $p\mid\Theta(v)$ and $L\ne0$, in
the coordinates of Lemma~\ref{lem:valuations}. Then $h_1(b)=\lambda$ is a
nonzero constant, and $h_2(b)=\lambda q(b)$ for a polynomial $q$ of
degree at most two with $[b^2]q=h_{\infty,2}/\lambda$; this coefficient
is nonzero if $e_\infty=2$ and zero if $e_\infty\ge3$. After a translation, which
preserves $\Theta(v)$, the augmentation and the cofactor ideals,
$q(b)=ab^2$ with $a=h_{\infty,2}/\lambda\ne0$ when $e_\infty=2$, and $q=0$ when
$e_\infty\ge3$.
\end{lemma}

\begin{proof}
By Lemma~\ref{lem:valuations}, $p\mid s$ and every $f_\delta$ lies in
$\p$. The projections of $v$
satisfy~\eqref{eq:reconstruction-congruence} by
Lemma~\ref{lem:reconstruct}, and Lemma~\ref{lem:moments} shows that the
expansions $H_\delta$ satisfy~\eqref{eq:degree} and~\eqref{eq:constant}. Write
$f_\infty=\varpi^{e_\infty}z_\infty$ with $z_\infty$ prime to $\p$. For $e_\infty<p$ the
expansion $H_\infty=(1-\exp T)^{e_\infty}z_\infty$ has $T$-adic order exactly
$e_\infty$, and for $e_\infty\ge p$ we have $H_\infty=0$. In either case
$h_{\infty,1}=0$, while $h_{\infty,2}$ is nonzero for $e_\infty=2$ and zero for
$e_\infty\ge3$. The conditions~\eqref{eq:degree} with $n=1$ now show that
$h_1(b)=\lambda$ is a constant, which is nonzero because
$v_\p(f_b)=1$. With $n=2$ they give $\deg h_2\le2$ and
$[b^2]h_2=h_{\infty,2}$, which are the assertions about $q$.

A translation by $g=(g_1,g_2)\in G$ preserves $\Theta(v)$
(Section~\ref{sec:characters}), permutes the entries of $v$ and multiplies
each $f_\delta$ by a root of unity, so it preserves the augmentation and
every cofactor ideal. It leaves $h_1$ unchanged and adds $(g_1+bg_2)\lambda$ to
$h_2(b)$, that is, $g_1+bg_2$ to $q(b)$. Choosing $g_1$ and $g_2$ to cancel the
affine part of $q$ leaves $q(b)=([b^2]q)\,b^2$, which is $ab^2$ if $e_\infty=2$
and $0$ if $e_\infty\ge3$.
\end{proof}

\begin{lemma}\label{lem:degree-six}
Let $p\ge7$, let $a\in\F^\times$, and let $g_2,g_3\in\F[b]$ satisfy
$\deg g_2\le3$ and $\deg g_3\le4$. Then
\[
 D(b)=[T^3]\log(1+ab^2T+g_2(b)T^2+g_3(b)T^3)
\]
is a polynomial of degree exactly six, with
$[b^6]D=a^3/3$ and $[b^5]D=-a[b^3]g_2$. It has at most six zeros in
$\F$.
\end{lemma}

\begin{proof}
By~\eqref{eq:cubic-coordinate},
$D=g_3-ab^2g_2+a^3b^6/3$. The first two terms have degrees at most
four and five, respectively, whereas $a^3/3\ne0$. Since $6\le p-1$,
this polynomial also represents $D$ as a function on $\F$.
\end{proof}

\begin{lemma}[Cofactor bound]\label{lem:support}
Let $\Theta(v)=\pm p^Nm$ with $p\nmid m$ and $1\le N\le2p+2$.
In the coordinates of Lemma~\ref{lem:valuations} put $t=e_\infty$ and
$I_b=(f_b/\varpi)R$ for $b\in\F$. Suppose $2\le t\le p-2$.
If $t=2$, take $(j,d)=(3,6)$.
If $3\le t\le p-2$, take any positive integer $\nu$ with
\[
 j=\nu(t-1),\qquad d=\nu t\le p-1.
\]
After the translation of Lemma~\ref{lem:normal-form}, the coefficient
of $T^j$ in $\log(H_b/(\lambda T))$ has degree exactly $d$ in $b$.
If $j$ is odd and at least three, or $j=2i$ with $i\in\Irru$, then this
coefficient vanishes wherever $I_b=R$. At least $p-d$ of the $I_b$
are then proper, and
\begin{equation}\label{eq:cofactor-bound}
 |m|\ge(p+1)^{p-d},\qquad
 \sum_{\ell\mid m}\left\lfloor\frac{v_\ell(m)}{\ord_p(\ell)}\right\rfloor
 \ge p-d,
\end{equation}
where the sum is over rational primes.
\end{lemma}

\begin{proof}
By Lemma~\ref{lem:valuations}, $L\ne0$ and $f_b=\varpi z_b$ with
$z_b$ prime to $\p$ for $b\in\F$. Write
\[
 \widetilde H_b=H_b/(\lambda T)=1+\sum_{k=1}^{p-2}g_k(b)T^k
 \pmod{T^{p-1}},\qquad g_k=h_{k+1}/\lambda.
\]
When $t=2$, Lemma~\ref{lem:normal-form} gives $g_1=ab^2$ with $a\ne0$.
Since $\deg g_2\le3$ and $\deg g_3\le4$ by~\eqref{eq:degree},
Lemma~\ref{lem:degree-six} gives degree six, with leading coefficient
$a^3/3$, for $[T^3]\log\widetilde H_b$.

Now let $3\le t\le p-2$. We have $g_1=0$ after translation.
The moment conditions~\eqref{eq:degree} apply through $h_{j+1}$.
Indeed, $j+1=t\le p-2$ if $\nu=1$, and $j+1=d-\nu+1\le p-2$ if $\nu\ge2$.
Thus $\deg g_k\le k+1$ for $1\le k\le j$, with
$[b^{k+1}]g_k=h_{\infty,k+1}/\lambda$. The exceptional expansion starts
in degree $t$, so $\deg g_k\le k$ for $k<t-1$ and
$[b^t]g_{t-1}=h_{\infty,t}/\lambda\ne0$.

Form $[T^j]\log\widetilde H_b$ in $\F[b]$ before reducing modulo $b^p-b$.
A product $g_{k_1}\cdots g_{k_s}$ in this coefficient has
$k_1+\cdots+k_s=j$ with all $k_l\ge1$. Its degree is at most $j+w$,
where $w$ counts the indices $k_l\ge t-1$. We have $w\le\nu$, and
equality uses the whole sum $j$ in $\nu$ indices equal to $t-1$.
Hence only $(-1)^{\nu+1}g_{t-1}^{\nu}/\nu$ contributes to $b^d$.
Its coefficient is
\[
 \frac{(-1)^{\nu+1}}{\nu}
 \left(\frac{h_{\infty,t}}{\lambda}\right)^\nu\ne0,
\]
since $\nu<p$. As $d\le p-1$, the polynomial of degree $d$ already
represents the coefficient as a function on $\F$.

In both cases, \eqref{eq:relative-log} with $e=1$ identifies the odd
coefficient at $j\ge3$ with $c_j(z_b)$ and the even coefficient at
$j=2i$ with $\kappa_i+c_{2i}(z_b)$. When $I_b=R$, the element $z_b$
is a global unit. Its odd coordinates of degree at least three vanish
by Lemma~\ref{lem:units}(i). If $i\in\Irru$, its coordinate $c_{2i}$
vanishes by definition, and $\kappa_i=0$ since $\Irru\subseteq\Irr$.
In the stated degrees there are at most $d$ roots, so at least $p-d$
finite directions have proper cofactor ideals.

By~\eqref{eq:factor}, $\prod_{b\in\F}\Nm(I_b)$ divides $|m|$. Each
proper $I_b$ has a prime ideal factor over some $\ell\ne p$ whose
norm is $\ell^{\ord_p(\ell)}\equiv1\pmod p$, hence at least $p+1$.
Counting these factors with multiplicity, at most
$\lfloor v_\ell(m)/\ord_p(\ell)\rfloor$ can lie over any one $\ell$.
This proves both bounds in~\eqref{eq:cofactor-bound}.
\end{proof}

For the next three lemmas, take a vector $v\in\Z^G$ with
$\Theta(v)\ne0$ and $1\le N=v_p(\Theta(v))\le2p+2$, and suppose that
\begin{equation}\label{eq:unit-character-hypotheses}
 f_\delta=\varpi^{e_\delta}z_\delta,\qquad
 z_\delta\in R^\times\quad(\delta\in\mathbf P^1(\F)).
\end{equation}
By Lemma~\ref{lem:valuations}, $L\ne0$, and in suitable coordinates $e_b=1$
for $b\in\F$ and $e_\infty\ge2$. Let $\lambda$ and $q$ be as in
Lemma~\ref{lem:normal-form}, and put
\[
 \widetilde H_b(T)=H_b(T)/(\lambda T)=1+q(b)T+\cdots\pmod{T^{p-1}}.
\]
By~\eqref{eq:relative-log} and Lemma~\ref{lem:units}(i),
\begin{equation}\label{eq:generic-log}
 \log\widetilde H_b(T)=\log E(T)+\log\frac{z_b(T)}{z_b(0)}
\end{equation}
has no odd terms of degree $3\le j\le p-2$, and its coefficient of
$T^{2i}$ is $\kappa_i+c_{2i}(z_b)$.

\begin{lemma}\label{lem:val-two}
Under the hypotheses~\eqref{eq:unit-character-hypotheses} above,
$e_\infty\ne2$.
\end{lemma}

\begin{proof}
If $e_\infty=2$, Lemma~\ref{lem:support} gives at least $p-6>0$
proper finite cofactor ideals, contrary to~\eqref{eq:unit-character-hypotheses}.
\end{proof}

\begin{lemma}\label{lem:parity}
Under the hypotheses~\eqref{eq:unit-character-hypotheses} above,
suppose that $e_\infty\ge3$ and $q=0$. Then $h_n(b)=0$ for every $b\in\F$ and
every even $n$ with $2\le n\le p-1$, and $h_{\infty,n}=0$ for
every even $n$ with $2\le n\le p-2$.
Consequently $e_\infty$ is odd or $e_\infty\ge p-1$. If moreover $e_\infty\le p-4$, then
$p^2\mid s$.
\end{lemma}

\begin{proof}
Since $q=0$, the linear coefficient of $\log \widetilde H_b$ vanishes, and
by~\eqref{eq:generic-log} so do its odd coefficients of degrees $3$ to
$p-2$. Hence $\widetilde H_b$ is even modulo $T^{p-1}$, and
$h_n(b)=\lambda[T^{n-1}]\widetilde H_b=0$ for every even $n\le p-1$;
by~\eqref{eq:degree}, $h_{\infty,n}=[b^n]h_n=0$ for every even
$n\le p-2$. If $e_\infty\le p-2$, then $H_\infty=(1-\exp T)^{e_\infty}z_\infty$ has
$T$-adic order exactly $e_\infty$, so $h_{\infty,e_\infty}\ne0$ and $e_\infty$ is odd.

Now let $e_\infty\le p-4$, so that $e_\infty$ is odd. The series
$J=H_\infty/(h_{\infty,e_\infty}T^{e_\infty})$, taken modulo $T^{p-e_\infty}$, has coefficient
$h_{\infty,e_\infty+k}/h_{\infty,e_\infty}$ at $T^k$, and this coefficient vanishes
for odd $k\le p-2-e_\infty$ because $e_\infty+k$ is then even and at most $p-2$. Thus
$J=J_0+\mu T^{p-1-e_\infty}$ with $J_0$ even and
$\mu=h_{\infty,p-1}/h_{\infty,e_\infty}$. As $p-1-e_\infty$ is odd and is the largest
degree below $p-e_\infty$, $\mu$ is the coordinate of degree $p-1-e_\infty\ge3$ of
$\log J$. This coordinate vanishes by~\eqref{eq:relative-log} and
Lemma~\ref{lem:units}(i), and we obtain $h_{\infty,p-1}=0$. Since also
$h_{p-1}(b)=\lambda[T^{p-2}]\widetilde H_b=0$, the augmentation
condition~\eqref{eq:constant} reads $s/p\equiv0\pmod p$.
\end{proof}

\begin{lemma}[Bernoulli obstruction]\label{lem:obstruction}
Under the hypotheses~\eqref{eq:unit-character-hypotheses} above,
suppose that $e_\infty=2r+1\le p-2$ and $q=0$. Then every $i\in\Irru$ with
$r\mid i$ satisfies $(2r+1)i>r(p-1)$.
\end{lemma}

\begin{proof}
Let $i\in\Irru$ with $r\mid i$. If $(2r+1)i/r\le p-1$, take
$t=2r+1$ and $\nu=i/r$ in Lemma~\ref{lem:support}. Then $j=2i$
and $d=(2r+1)i/r$, so at least $p-d\ge1$ finite cofactor ideals
are proper, contrary to~\eqref{eq:unit-character-hypotheses}.
\end{proof}

\begin{proposition}\label{prop:construction}
Let $r\ge1$ be admissible for $\Irru$, with $2r+1\le p-2$. Then
$p^{p+3+2r}\Z\subseteq S(C_p\times C_p)$. If $r$ is admissible for
$\Irr$, then for every nonzero integer $n$ there is $v\in\Z^G$ with
$\Theta(v)=p^{p+3+2r}n$ whose character values $f_\delta$, $\delta\in\mathbf P^1(\F)$,
are powers of $\varpi$ times elements of the group $C$ of
Lemma~\ref{lem:units}(iii).
\end{proposition}

\begin{proof}
Let $\mathcal I$ be $\Irru$ or $\Irr$, according to the case, and put
$V=\{\nu\ge1:r\nu\in\mathcal I\}$. By admissibility, $(2r+1)\nu>p-1$
for every $\nu\in V$; in particular $1\notin V$, because $2r+1\le p-2$.
Every $\nu\in V$ satisfies $\nu\le M<p$. Let
\[
 P(X)=1+X+\sum_{\nu\in V}\alpha_\nu X^\nu\in\F[X],
\]
where the coefficients $\alpha_\nu$ are chosen recursively, in
increasing order of $\nu$, so that $[X^\nu]\log P=0$ for every
$\nu\in V$. This is possible because $[X^\nu]\log P$ equals $\alpha_\nu$
plus an expression in the coefficients of lower degree, and the
denominators in the logarithm are smaller than $p$.

Fix a nonzero integer $n$ and prescribe
\begin{equation}\label{eq:tail-data}
 \begin{gathered}
 H_b(T)=T\,P(b^{2r+1}T^{2r})\pmod{T^p}\quad(b\in\F),\\
 H_\infty(T)=T^{2r+1},\qquad s=p^2n.
 \end{gathered}
\end{equation}
The only nonzero coefficients of these expansions are $h_1=1$,
$h_{2r+1}(b)=b^{2r+1}$, $h_{\infty,2r+1}=1$, and
$h_{2r\nu+1}(b)=\alpha_\nu b^{(2r+1)\nu}$ for $\nu\in V$. The first
three satisfy~\eqref{eq:degree} because $2r+1\le p-2$. For $\nu\in V$
we have $(2r+1)\nu>p-1$, and on $\F$ the monomial $b^{(2r+1)\nu}$
therefore equals $b^j$ with $j\equiv(2r+1)\nu\pmod{p-1}$ and
$1\le j\le(2r+1)\nu-(p-1)=2r\nu+1-(p-\nu)<2r\nu+1$. The function
$h_{2r\nu+1}$ has degree less than $2r\nu+1$, and its coefficient of
$b^{2r\nu+1}$ is zero, in agreement with $h_{\infty,2r\nu+1}=0$. Every
expansion in~\eqref{eq:tail-data} is odd, so all coefficients of degree
$p-1$ vanish; as $s/p=pn\equiv0\pmod p$, the augmentation
condition~\eqref{eq:constant} holds as well.

For $b\in\F$, take
\[
 z_b(T)=-P(b^{2r+1}T^{2r})/E(T)\pmod{T^{p-1}}.
\]
Since $1-\exp T=-TE(T)$, this residue gives~\eqref{eq:tail-data}, and
\[
 \log\frac{z_b(T)}{z_b(0)}=\log P(b^{2r+1}T^{2r})-\log E(T)
\]
has no odd terms of degree at least three. For $i\in\mathcal I$ its coefficient of $T^{2i}$
is $-\kappa_i=0$ if $r\nmid i$, and
$[X^{i/r}]\log P\cdot b^{(2r+1)i/r}-\kappa_i=0$ if $r\mid i$, because
then $i/r\in V$. By Lemma~\ref{lem:units}(iii) or~(iv), the residue $z_b(T)$ is realized by a unit
$z_b$ (a circular unit when $\mathcal I=\Irr$). For
$\delta=\infty$ the expansion $T^{2r+1}$ requires a unit residue congruent to
$-E(T)^{-(2r+1)}$ modulo $T^{p-1-2r}$, and we take
$z_\infty(T)=-E(T)^{-(2r+1)}$ modulo $T^{p-1}$. The series
$\log(z_\infty(T)/z_\infty(0))=-(2r+1)\log E(T)$ has no odd terms of degree at least three, and its
coefficient of $T^{2i}$ is $-(2r+1)\kappa_i=0$ for $i\in\mathcal I$.
There is therefore a unit $z_\infty$ with residue $z_\infty(T)$, again
circular when $\mathcal I=\Irr$. By the definition~\eqref{eq:H}, the
values $f_b=\varpi z_b$ and $f_\infty=\varpi^{2r+1}z_\infty$ have the
expansions~\eqref{eq:tail-data}.

By Lemma~\ref{lem:lifting} there is $v\in\Z^G$ with augmentation $s$
and these character values. Units have norm one and $\Nm(\varpi)=p$, so
\[
 \Theta(v)=s\cdot p^p\cdot p^{2r+1}=p^{p+3+2r}n.
\]
The value $0$ is the determinant of the zero vector.
\end{proof}

For example, for $r=1$ and $p=37,59,67$, the polynomial $P$ is
$1+X+X^{16}/16$, $1+X+X^{22}/22$ and $1+X-X^{29}/29$, respectively, and
Proposition~\ref{prop:construction} gives $p^{p+5}\Z\subseteq S(C_p\times C_p)$
although $p$ is irregular.

\begin{lemma}\label{lem:admissible-range}
For every prime $p\ge7$, one has $1\le r^u\le r_p\le(p-5)/2$.
More precisely, if
$p\ge11$, put $k=\lfloor M/2\rfloor$ and
$R_0=\lceil3k/2\rceil$. Then
\begin{equation}\label{eq:admissible-range}
 r_p\le R_0\le M-1,\qquad
 p+3+2R_0\le\left\lfloor\frac{7(p+1)}4\right\rfloor.
\end{equation}
\end{lemma}

\begin{proof}
Since $\Irru\subseteq\Irr$, every integer admissible for $\Irr$ is
admissible for $\Irru$, so $1\le r^u\le r_p$.
Suppose first that $p\ge11$, so $k\ge2$ and
$k<R_0\le2k-1\le M-1$. If every integer $r$ with $k<r\le R_0$
were inadmissible, its only positive multiple at most $M$ being $r$
itself, each such $r$ would lie in $\Irr$. Thus $B_{2r}\equiv0\pmod p$
for $k<r\le R_0$. In the following calculation all series coefficients
used have degree at most $2R_0+2\le4k\le p-3$, so their denominators
are prime to $p$ and reduction modulo $p$ is defined. Let
\[
 B(X)=\frac{X}{\exp X-1},\qquad
 B_{\le2k}(X)=\sum_{j=0}^{2k}\frac{B_j}{j!}X^j\in\F[X].
\]
These congruences and $B_n=0$ for odd $n\ge3$ give
\[
 B(X)-B_{\le2k}(X)=O(X^{2R_0+2}),\qquad
 (\exp X-1)B_{\le2k}(X)-X=O(X^{2R_0+3}).
\]
Put $W(z)=\sum_{j=0}^{2k}(B_j/j!)(z)_j$, where
$(z)_j=z(z-1)\cdots(z-j+1)$ and $(z)_0=1$. For every integer
$0\le n\le2R_0+2<p$,
\[
 W(n)=n![X^n](\exp X)B_{\le2k}(X).
\]
Since $B_{\le2k}$ has degree at most $2k$, the second series identity implies that
$W(n)=0$ for $2k<n\le2R_0+2$. For odd $3\le n\le2k$, the identity
$(\exp X)B(X)=B(X)+X$ instead gives $W(n)=B_n=0$.
These two disjoint sets contain
\[
 (2R_0+2-2k)+(k-1)=2R_0+1-k\ge2k+1
\]
distinct elements of $\F$. This contradicts $\deg W\le2k$ and
$W(0)=1$. Some $k<r\le R_0$ is therefore admissible, proving
$r_p\le R_0\le M-1$. Finally, $2R_0\le3k+1$ and
$k\le(p-3)/4$ give
$p+3+2R_0\le p+4+3k\le7(p+1)/4$; the left side is an integer.
For $p=7$, the Bernoulli numbers $B_2=1/6$ and $B_4=-1/30$ show
that $7$ is regular, so $r_7=1=(7-5)/2$.
\end{proof}

\begin{proof}[Proof of Theorem~A]
By Lemma~\ref{lem:admissible-range}, $1\le r^u\le r_p\le(p-5)/2$,
so Proposition~\ref{prop:construction} applies with $r=r^u$; this proves
part~(a), the inclusion for $r_p$ following from $r^u\le r_p$.

For part~(b), suppose that $\Theta(v)=\pm p^Nm$. By
Lemma~\ref{lem:unit-cofactors} every character value is a unit times a
power of $\varpi$, and $e_\infty=N-v_p(s)-p\le N-p-1\le p-3$.
Lemma~\ref{lem:val-two} excludes $e_\infty=2$, so after a translation
$q=0$, and by Lemma~\ref{lem:parity} $e_\infty=2r+1$ is odd; hence
$e_\infty\le p-4$ and $v_p(s)\ge2$. By Lemma~\ref{lem:obstruction}, $r$
is admissible for $\Irru$, so $N=v_p(s)+p+2r+1\ge p+3+2r^u$, contrary to
the hypothesis on $N$.

For part~(c), take $N=p+2+2r^u$. Then $p+3\le N\le2p-3$ by
Lemma~\ref{lem:admissible-range}, so part~(b) with $m=1$ excludes $p^N$, while
part~(a) supplies every multiple of $p^{N+1}$. Thus
$e(p)=p+3+2r^u$. The uniform bound follows from
Lemma~\ref{lem:admissible-range} when $p\ge11$; for $p=7$,
$e(7)=12\le14$. Finally, if $p\nmid h_p^+$, then
$\Irru=\Irr$ by Lemma~\ref{lem:units}(v), proving part~(d).
\end{proof}

Since $r^u\ge1$, Theorem~A(b) applies to $N=p+3$ and
$N=p+4$ for every $p\ge7$; in particular
$\pm p^{p+3},\pm p^{p+4}\notin S(C_p\times C_p)$.

\begin{proof}[Proof of Corollary~\ref{cor:intro-p5}]
Since $p\nmid h_p^+$, we have $r^u=r_p$ by Theorem~A(d).
If every irregular pair $(p,2i)$ satisfies $3i>p-1$, then $r_p=1$, and
$p^{p+5}\Z\subseteq S(C_p\times C_p)$ by Theorem~A(a). If some
irregular pair has $3i\le p-1$, then $r_p\ge2$; as $p+5\le2p-2$ for
$p\ge7$, Theorem~A(b) shows that $p^{p+5}\notin S(C_p\times C_p)$.
\end{proof}

\begin{corollary}\label{cor:prime-exclusion}
For every prime $p\ge11$ and every rational prime $\ell\ne p$,
$\pm p^{p+3}\ell\notin S(C_p\times C_p)$.
\end{corollary}

\begin{proof}
By Lemma~\ref{lem:support}, at least $p-6\ge5$ of the ideals $I_b$
would be proper. Their norms exceed one and their product divides
$\ell$ by~\eqref{eq:factor}, which is impossible for a prime.
\end{proof}

\begin{proof}[Proof of Corollary~\ref{cor:least}]
The cofactor bound is Lemma~\ref{lem:support}. Now let $p\ge11$.
Put $N=v_p(\Theta)$; then $N\ge p+3$ by Lemma~\ref{lem:valuations}. If
$N=p+3$, then $|\Theta|\ge p^{p+3}(p+1)^{p-6}$ by Lemma~\ref{lem:support},
which exceeds $p^{p+5}$ for $p\ge11$. If $N=p+4$, then $N<p+3+2r^u$
because $r^u\ge1$, and Theorem~A(b) shows that the cofactor is divisible
by $\ell^{\ord_p(\ell)}$ for some prime $\ell$; this number is congruent
to $1$ modulo $p$ and larger than one, so $|\Theta|\ge p^{p+4}(p+1)$. If
$N\ge p+5$, then $|\Theta|\ge p^{p+5}$, with equality only for
$\Theta=\pm p^{p+5}$. If
$1$ is admissible for $\Irru$, then $p^{p+5}\in S(C_p\times C_p)$ by
Theorem~A(a). Otherwise $r^u\ge2$, and
Theorem~A(b) with $N=p+5$ excludes $\pm p^{p+5}$. The
remaining assertions follow from Theorem~A(d).
\end{proof}

\begin{proof}[Proof of Corollary~\ref{cor:minimum}]
Write $e=e(p)=p+3+2r$. By Theorem~A(a), $p^e$ occurs. Suppose
$0<\Theta=p^Nm<p^e$ with $N=v_p(\Theta)\ge1$. Then $m$ is positive
and prime to $p$, and $p+3\le N\le e-1$ by Lemma~\ref{lem:valuations}.
Lemma~\ref{lem:admissible-range} gives $2r\le p-5$, so $e\le2p-2$.
At $N=e-1$ we would have $m<p$, whereas Theorem~A(b) requires a
prime ideal norm $\ell^{\ord_p(\ell)}\ge p+1$ to divide $m$.

It remains to consider $p+3\le N\le e-2$. By~\eqref{eq:profile-valuation},
$N=v_p(s)+p+t\ge p+t+1$, where $2\le t\le2r$. At least
$k=2+2r-t$ proper finite cofactor ideals would give
\[
 \Theta\ge p^{p+t+1}(p+1)^k>p^{p+t+1+k}=p^e.
\]
We show that each possible $t$ gives this many ideals.

If $t=2$, Lemma~\ref{lem:support} gives $p-6$ proper ideals.
For $p\ge19$, Theorem~A(c) gives $2r\le(3p-5)/4\le p-6$.
The primes $11,13,17$ are regular~\cite[p.~410]{Washington1997},
so $r=1$ there and again $p-6\ge2r=k$.
For even $t\ge4$, take $\nu=1$ in Lemma~\ref{lem:support}.
Then $j=t-1$ is odd and at least three, and the lemma gives
$p-t\ge2+2r-t$ proper ideals. Finally, for odd $t=2a+1$ we have
$1\le a<r$. Choose $i$ as in~\eqref{eq:minimum-test} and take
$\nu=i/a$. Now $j=2i$ and $d=(2a+1)i/a\le p-2(r-a)-1$, so
$p-d\ge2(r-a)+1=2+2r-t$. This excludes every smaller value.
\end{proof}

\begin{corollary}\label{cor:p103}
Among odd primes, $103$ is the least $p$ for which
$p^{p+5}\Z\nsubseteq S(C_p\times C_p)$. We have $e(103)=116$, and the
least positive $103$-divisible determinant of $C_{103}\times C_{103}$
is $103^{116}$.
\end{corollary}

\begin{proof}
The only irregular pair at $103$ is $(103,24)$, and Vandiver's
conjecture holds there~\cite{HartHarveyOng2017}. Hence $\Irru=\{12\}$
by Theorem~A(d). Each of $a=1,2,3,4$ divides $12$. The corresponding
degrees $(2a+1)\cdot12/a$ are $36,30,28,27$, respectively. All are at
most $102$, so none of these $a$ is admissible. Since $5\nmid12$, we
have $r=5$ and $e(103)=103+3+10=116$. The four degrees also satisfy
the stronger bounds $94,96,98,100$ in~\eqref{eq:minimum-test}, so
Corollary~\ref{cor:minimum} gives the asserted minimum. In fact, at
valuations $106\le N\le114$ the cubic, even and odd cases above all
give at least $103-36=67$ proper finite ideals. At $N=115$ a cofactor
below $103$ is excluded by Theorem~A(b).

The irregular primes below $103$ are $37$, $59$, $67$ and $101$, with
irregular pairs $(37,32)$, $(59,44)$, $(67,58)$ and
$(101,68)$~\cite[p.~410]{Washington1997}. Each satisfies $3i>p-1$,
so Theorem~A(a) gives $p^{p+5}\Z\subseteq S(C_p\times C_p)$ for all
$7\le p<103$. For $p=3,5$, even $p^{p+3}\Z$ lies in the determinant
set~\cite[Theorem~6.1]{PinnerSmyth2020}, \cite{Panraksa2026Order25}.
\end{proof}

\section{Regular primes: the \texorpdfstring{$p$}{p}-adic valuations \texorpdfstring{$p+3$}{p+3} and \texorpdfstring{$p+4$}{p+4}}\label{sec:systems}

In this section and the next two, $p\ge7$ is a regular prime. For fixed
cofactor ideals, Lemma~\ref{lem:fixed-ideal} identifies which
logarithmic coordinates can vary with the choice of generators.
We use those coordinates to test whether the given ideals can occur
in a determinant of valuation $p+3$ or $p+4$.

Fix nonzero principal integral ideals
$I_b=y_bR$ for $b\in\F$ and $I_\infty=y_\infty R$, all prime to $\p$,
and write
\[
 D_j(b)=c_j(y_b),\qquad d_{\infty,j}=c_j(y_\infty)\qquad(j=3,5,\ldots,p-2).
\]
The odd coordinates of global units vanish
(Lemma~\ref{lem:units}(i)), so these coordinates do not depend on the chosen
generators. As in Section~\ref{sec:characters}, we compute with functions of $b$ in
$\mathcal B=\F[b]/(b^p-b)$, and degrees refer to the representative of
degree at most $p-1$.

Fix an integer $s\ne0$ divisible by $p$ and put $\tau=s/p\bmod p$. We
seek character values
\[
 f_b=\varpi y_bu_b\quad(b\in\F),\qquad f_\infty=\varpi^{e_\infty}y_\infty u_\infty,
\]
with $e_\infty\in\{2,3\}$, global units $u_b,u_\infty$, and $[T]H_b=1$ for every $b\in\F$. Every determinant of $p$-adic valuation $p+3$ or $p+4$
arises in this way. After the change of coordinates of
Lemma~\ref{lem:valuations}, Lemma~\ref{lem:normal-form} shows that
$h_1(b)=\lambda$ for a constant $\lambda\ne0$. Replacing $v$ by
$(i,j)\mapsto v_{\lambda i,\lambda j}$ replaces each character value by
its image under $\sigma_{\lambda^{-1}}$. This makes the leading
coefficients equal to one and replaces the cofactor ideals by Galois
conjugates.

With $\lambda=1$, write $H_b=T(1+q(b)T+\cdots)$. By
Lemma~\ref{lem:normal-form} we may assume, after a translation, that
\begin{equation}\label{eq:affine-q-normal}
 q(b)=ab^2\quad(e_\infty=2),\qquad q(b)=0\quad(e_\infty=3).
\end{equation}
For $e_\infty=2$, the coefficient $a=h_{\infty,2}$ is the leading coefficient
of the exceptional expansion.

\begin{definition}\label{def:operator}
Let $2\le w\le p-1$, let $q\in\F$, and let $d_j\in\F$ for odd
$j$ with $3\le j<w$. Set
\[
 \Psi_{q,d}(T)=\exp\Bigl(qT+\sum_{\substack{3\le j<w\\ j\text{ odd}}}d_jT^j\Bigr)
 \pmod{T^w}.
\]
For given even coefficients $x_{2i}$ with $2\le2i<w$, define
$\varepsilon_0=1$ and, recursively for $i\ge1$,
\begin{equation}\label{eq:affine-recurrence}
 \varepsilon_{2i}=x_{2i}-\sum_{0\le j<i}\varepsilon_{2j}\,[T^{2i-2j}]\Psi_{q,d}.
\end{equation}
Then put
\begin{equation}\label{eq:affine-operator}
 \Aop_w(q,d;x)=\Psi_{q,d}(T)\sum_{0\le2i<w}\varepsilon_{2i}T^{2i}\pmod{T^w}.
\end{equation}
The same definition applies over $\mathcal B$.
\end{definition}

\begin{lemma}\label{lem:affine-triangular}
The series $\Aop_w(q,d;x)$ is the unique series modulo $T^w$ with
constant coefficient one, first logarithmic coordinate $q$, odd
logarithmic coordinates $d_j$ for $3\le j<w$, and even coefficients
$x_{2i}$. For fixed $q$ and $d$, each of its coefficients is an affine
function of the $x_{2i}$: the coefficient of degree $2i$ is $x_{2i}$, and
the coefficient of degree $2i+1$ depends only on $x_2,\ldots,x_{2i}$.
\end{lemma}

\begin{proof}
All factorials and logarithmic denominators that occur are smaller than
$p$. The truncated exponential and logarithm are therefore inverse
bijections between $T\F[T]/(T^w)$ and $1+T\F[T]/(T^w)$, and similarly
over $\mathcal B$. A series with the prescribed odd logarithmic
coordinates, divided by $\Psi_{q,d}$, has an even logarithm, so it is
even. Conversely, every even series with constant coefficient one has an
even logarithm. It follows that the series with the prescribed
logarithmic coordinates are exactly the products $\Psi_{q,d}\mathcal E$, where
$\mathcal E=\sum_{0\le2i<w}\varepsilon_{2i}T^{2i}$ is even and
$\varepsilon_0=1$.
Matching the coefficient of degree $2i$ determines $\varepsilon_{2i}$
and gives~\eqref{eq:affine-recurrence}; by induction, $\varepsilon_{2i}$
is affine in $x_2,\ldots,x_{2i}$. The coefficient of degree $2i+1$ of the
product involves only $\varepsilon_0,\ldots,\varepsilon_{2i}$. This
proves uniqueness and the triangular dependence.
\end{proof}

The unknowns for the directions $b\in\F$ are polynomials
\begin{equation}\label{eq:affine-unknowns}
 x_{2i}\in\F[b]_{\le2i+1}\qquad(1\le i\le M),
\end{equation}
where $\F[b]_{\le k}$ denotes the space of polynomials of degree at most
$k$. For fixed $q$ and $D$ we form
\begin{equation}\label{eq:affine-generic}
 H_b(T)=T\,\Aop_{p-1}\bigl(q(b),D(b);x_2(b),x_4(b),\ldots,x_{p-3}(b)\bigr),
 \qquad h_n(b)=[T^n]H_b.
\end{equation}
By~\eqref{eq:degree}, the coefficient of $T^n$ in the exceptional
expansion must be $[b^n]h_n$ for $1\le n\le p-2$. The
coefficients $h_{2i+1}=x_{2i}$ satisfy the degree bounds
of~\eqref{eq:degree} by~\eqref{eq:affine-unknowns}, and so do $h_1=1$
and $h_2=q$. The remaining degree conditions on the expansions $H_b$
are exactly
\begin{equation}\label{eq:affine-generic-equations}
 [b^k]h_n=0\qquad(n=4,6,\ldots,p-3,\quad n<k\le p-1).
\end{equation}
These equations are affine in
the coefficients of the unknowns~\eqref{eq:affine-unknowns}.

By Lemma~\ref{lem:affine-triangular}, the logarithm of
$H_b/T=1+qT+x_2T^2+h_4T^3+\cdots$ has coefficient $D_3$ at $T^3$.
By~\eqref{eq:cubic-coordinate},
\begin{equation}\label{eq:affine-h4}
 h_4=D_3+qx_2-q^3/3.
\end{equation}
If $q=ab^2$ and $\deg h_4\le4$, then Lemma~\ref{lem:degree-six},
with $g_2=x_2$ and $g_3=h_4$, gives
\begin{equation}\label{eq:affine-cubic}
 \deg D_3=6,\qquad a^3=3[b^6]D_3,\qquad [b^3]x_2=-[b^5]D_3/a.
\end{equation}
Thus $[b^3]h_3=a\eta$ with $\eta=-[b^5]D_3/a^2$. If $q=0$,
\eqref{eq:affine-h4} instead gives the necessary condition $\deg D_3\le4$.

\begin{definition}\label{def:affine-tests}
Condition $\Ltwo(D,d_\infty)$ requires that $\deg D_3=6$ and that, for
some $a\in\F^\times$ with $a^3=3[b^6]D_3$, the following system in the
unknowns~\eqref{eq:affine-unknowns} is consistent. Put
\[
 \eta=-[b^5]D_3/a^2,\qquad q(b)=ab^2,
\]
form the series~\eqref{eq:affine-generic}, and impose the
equations~\eqref{eq:affine-generic-equations} together with the
equations
\begin{equation}\label{eq:affine-r2-leading}
 [b^2]h_2=a,\qquad [b^3]h_3=a\eta.
\end{equation}
Then form
\begin{equation}\label{eq:affine-r2-Q}
 Q_\infty=\Aop_{p-3}\bigl(\eta,d_\infty;([b^{2i+2}]h_{2i+2}/a)_{1\le i\le M-1}\bigr)
\end{equation}
and impose the equations
\begin{equation}\label{eq:affine-r2-special}
 [b^{j+2}]h_{j+2}=a\,[T^j]Q_\infty\qquad(j=3,5,\ldots,p-4).
\end{equation}

Condition $\Lthree(D,d_\infty,\tau)$ requires that $\deg D_3\le4$ and
that, for some $a\in\F^\times$, the following system in the
unknowns~\eqref{eq:affine-unknowns} is consistent. Put
\[
 \eta=[b^4]D_3/a,\qquad q=0,
\]
and impose~\eqref{eq:affine-generic-equations} together with the
equations
\begin{equation}\label{eq:affine-r3-leading}
 [b^3]h_3=a,\qquad [b^4]h_4=a\eta.
\end{equation}
Then form
\begin{equation}\label{eq:affine-r3-Q}
 Q_\infty=\Aop_{p-3}\bigl(\eta,d_\infty;([b^{2i+3}]h_{2i+3}/a)_{1\le i\le M-1}\bigr)
\end{equation}
and impose the equations
\begin{equation}\label{eq:affine-r3-special}
 \begin{split}
 [b^{j+3}]h_{j+3}&=a\,[T^j]Q_\infty\qquad(j=3,5,\ldots,p-6),\\
 \tau+\sum_{b\in\F}h_{p-1}(b)+a\,[T^{p-4}]Q_\infty&=0.
 \end{split}
\end{equation}
For $p=7$ the first family in~\eqref{eq:affine-r3-special} is empty.
\end{definition}

For fixed $a$, every $[b^n]h_n$ is affine in the unknowns by
Lemma~\ref{lem:affine-triangular}, and the same lemma, applied
to~\eqref{eq:affine-r2-Q} or~\eqref{eq:affine-r3-Q}, shows that the
equations~\eqref{eq:affine-r2-special} and~\eqref{eq:affine-r3-special}
are affine as well. Only $\Lthree$ involves $\tau$: for $e_\infty=2$ the
series $Q_\infty$ is defined modulo $T^{p-3}$ and does not constrain
$h_{\infty,p-1}$, whereas for $e_\infty=3$ the coefficient
$h_{\infty,p-1}$ is $a[T^{p-4}]Q_\infty$.

\begin{remark}\label{rem:affine-count}
For each value of $a$, the system in Definition~\ref{def:affine-tests}
has $\sum_{i=1}^{M}(2i+2)=M(M+3)=(p^2-9)/4$ scalar unknowns, the
coefficients in~\eqref{eq:affine-unknowns}, and $M^2+1=(p-3)^2/4+1$
equations: $\sum_{i=2}^{M}(p-1-2i)=M(M-1)$
from~\eqref{eq:affine-generic-equations}, two
from~\eqref{eq:affine-r2-leading} or~\eqref{eq:affine-r3-leading}, one
of which holds identically, and $M-1$ from~\eqref{eq:affine-r2-special}
or~\eqref{eq:affine-r3-special}, including the last equation
of~\eqref{eq:affine-r3-special}. Since cubing on the cyclic group
$\F^\times$ has kernel of order $\gcd(3,p-1)$, the equation
$a^3=3[b^6]D_3$ has no solution or exactly $\gcd(3,p-1)$ solutions, so
$\Ltwo$ involves at most $\gcd(3,p-1)$ values of $a$; $\Lthree$ involves
all $p-1$ values $a\in\F^\times$. At $p=7$, for example, each system has ten unknowns, two equations
from~\eqref{eq:affine-generic-equations}, two from~\eqref{eq:affine-r2-leading} or~\eqref{eq:affine-r3-leading} and one
equation from~\eqref{eq:affine-r2-special} or~\eqref{eq:affine-r3-special}.
\end{remark}

\begin{proposition}\label{prop:placed}
Let $p\ge7$ be regular, let $e_\infty\in\{2,3\}$, and let the principal ideals
and the nonzero augmentation $s$ be as fixed above. There is a vector $v\in\Z^G$ with augmentation $s$ and character values
\[
 f_b=\varpi y_bu_b\quad(b\in\F),\qquad
 f_\infty=\varpi^{e_\infty}y_\infty u_\infty,
\]
where all $u_b,u_\infty$ are global units and all expansions $H_b$ have
leading coefficient one, if and only if $\Ltwo(D,d_\infty)$ holds when
$e_\infty=2$, or $\Lthree(D,d_\infty,s/p\bmod p)$ holds when $e_\infty=3$. Any solution of the corresponding system yields such a $v$, and its
determinant is
\begin{equation}\label{eq:affine-determinant}
 s\,p^{p+e_\infty}\prod_{b\in\F}\Nm(I_b)\,\Nm(I_\infty).
\end{equation}
\end{proposition}

\begin{proof}
Suppose first that $v$ exists. The translation
leading to~\eqref{eq:affine-q-normal} preserves the ideals
$I_\delta$, the leading coefficients $[T]H_b$ and the augmentation. By~\eqref{eq:fixed-ideal} and
Lemma~\ref{lem:affine-triangular}, the expansions $H_b$ have the
form~\eqref{eq:affine-generic}, and the bounds~\eqref{eq:degree} on the
coefficients $h_{2i+1}$ give exactly the polynomial
spaces~\eqref{eq:affine-unknowns}. The other degree bounds on the
expansions $H_b$ give~\eqref{eq:affine-generic-equations}, and the
exceptional coefficients of degree at most $p-2$ must equal the
corresponding $[b^n]h_n$.

Let $e_\infty=2$. The leading exceptional coefficient is $[b^2]h_2=a\ne0$.
Equation~\eqref{eq:affine-h4} and the bound $\deg h_4\le4$
give~\eqref{eq:affine-cubic}; in particular $a^3=3[b^6]D_3$ and
$[b^3]h_3=a\eta$. The
series $H_\infty/(aT^2)$ has first logarithmic coordinate $\eta$, odd
logarithmic coordinates $d_{\infty,j}$ for $3\le j\le p-4$, and even
coefficients $[b^{2i+2}]h_{2i+2}/a$ for $1\le i\le M-1$. By the uniqueness in
Lemma~\ref{lem:affine-triangular}, its truncation modulo $T^{p-3}$ is
the series~\eqref{eq:affine-r2-Q}. Matching its remaining odd
coefficients gives~\eqref{eq:affine-r2-special}, so $\Ltwo(D,d_\infty)$
holds.

Conversely, let $e_\infty=2$ and take a solution of the system for some $a$.
It gives all expansions $H_b$ and, through degree $p-2$, the
exceptional series $aT^2Q_\infty$. The coefficients of this series agree
with the required values $[b^n]h_n$: in degrees two and three
by~\eqref{eq:affine-r2-leading}, in higher even degrees by the inputs
of~\eqref{eq:affine-r2-Q}, and in higher odd degrees
by~\eqref{eq:affine-r2-special}; in degree one both vanish. Set
\begin{equation}\label{eq:affine-free-constant}
 h_{\infty,p-1}=-\tau-\sum_{b\in\F}h_{p-1}(b)
\end{equation}
and add $h_{\infty,p-1}T^{p-1}$ to $aT^2Q_\infty$. In $H_\infty/(aT^2)$
this changes only the coefficient of degree $p-3$, which is even. As
$p-3$ is the largest degree below $p-2$, no logarithmic coordinate of lower
degree changes, and no odd coordinate is prescribed in degree $p-3$.
The exceptional expansion therefore still has the
form~\eqref{eq:fixed-ideal} with $y=y_\infty$. Conditions~\eqref{eq:degree}
and~\eqref{eq:constant} now hold, for every value of $\tau$.

Let $e_\infty=3$. Then $q=0$, so~\eqref{eq:affine-h4} reads $h_4=D_3$ and
forces $\deg D_3\le4$. The first exceptional coefficient is $[b^3]h_3=a\ne0$,
and the next is $[b^4]h_4=[b^4]D_3=a\eta$. Modulo $T^{p-3}$, the series $H_\infty/(aT^3)$ has
first logarithmic coordinate $\eta$, odd logarithmic coordinates
$d_{\infty,j}$ for $3\le j\le p-4$, and even coefficients $[b^{2i+3}]h_{2i+3}/a$
for $1\le i\le M-1$, so by
Lemma~\ref{lem:affine-triangular} it coincides
with~\eqref{eq:affine-r3-Q}. Its odd coefficients below
degree $p-4$ give the first family of equations
in~\eqref{eq:affine-r3-special}. Its highest coefficient is
$h_{\infty,p-1}/a=[T^{p-4}]Q_\infty$, and the augmentation condition
becomes the last equation of~\eqref{eq:affine-r3-special}. This proves
necessity. Conversely, given a solution for some $a$, take
$H_\infty=aT^3Q_\infty$ modulo $T^p$. The
equations~\eqref{eq:affine-r3-leading}, the even inputs
of~\eqref{eq:affine-r3-Q} and the first family
in~\eqref{eq:affine-r3-special} match every coefficient through degree
$p-2$. The coefficients of degrees one and two vanish, as do $[b^1]h_1$ and
$[b^2]h_2$, and the last equation of~\eqref{eq:affine-r3-special} gives the
augmentation condition.

It remains to lift the expansions to character values. For a direction
whose character value is to have valuation $k_0\in\{1,2,3\}$, let $y$ be
the chosen generator of its cofactor ideal, and put
\begin{equation}\label{eq:affine-unit-recovery}
 z(T)=\frac{H(T)/T^{k_0}}{((1-\exp T)/T)^{k_0}}\pmod{T^{p-k_0}},
 \qquad \varphi(T)=\frac{z(T)}{y(T)}.
\end{equation}
These divisions are by units. By~\eqref{eq:relative-log}, division by
$((1-\exp T)/T)^{k_0}$ changes only the constant, linear and even
logarithmic coordinates. Since $H$ has the form~\eqref{eq:fixed-ideal},
every odd coordinate of $\varphi$ of degree at least three and less than
$p-k_0$ vanishes. Extend the logarithm of $\varphi$ to degree $p-2$ by zero in
the coordinates of degree at least $p-k_0$. By Lemma~\ref{lem:units}(iii)
(here $\Irr=\varnothing$), there is a unit $u$ with this residue in the
group generated by $\zeta$ and the circular units $u_a$. Then
$f=\varpi^{k_0}yu$ has the expansion $H$.

By Lemma~\ref{lem:lifting} there is $v\in\Z^G$ with augmentation $s$ and
these character values. Its determinant is
$s\prod_\delta\Nm(f_\delta)$, which equals~\eqref{eq:affine-determinant}
because global units have norm one, $\Nm(\varpi)=p$ and
$\Nm(y)=\Nm(yR)$.
\end{proof}

\section{Assignments and the proof of Theorem~B}\label{sec:assignments}

For a given cofactor $m>0$ prime to $p$, there are only finitely many
choices of the cofactor ideals. We obtain them by distributing prime
ideals above the prime factors of $m$ among the $p+1$ directions, then
apply the tests for valuations $p+3$ and $p+4$.
Choose a prime ideal $\mathfrak q_\ell$ above each prime $\ell\mid m$.
All prime ideals above $\ell$ are its Galois conjugates and have norm
$\ell^{\ord_p(\ell)}$.

\begin{definition}\label{def:assignment}
An \emph{assignment} for $m$ places, for each prime $\ell\mid m$, at
most $\lfloor v_\ell(m)/\ord_p(\ell)\rfloor$ conjugates of
$\mathfrak q_\ell$ in the $p+1$ directions, with repetitions allowed.
For each direction $\delta$, let $I_\delta$ be the product of the ideals
placed in $\delta$, an empty product being $R$. The assignment is
\emph{principal} if every $I_\delta$ is principal. Its coordinates are
$D_j(b)=c_j(I_b)$ and $d_{\infty,j}=c_j(I_\infty)$, with $c_j$ as
in~\eqref{eq:ideal-coordinates}. We write $D$ for the whole collection
$(D(b)_{b\in\F},d_\infty)$, so that $\Ltwo(D)$ and $\Lthree(D,\tau)$
denote the conditions of Definition~\ref{def:affine-tests}.
\end{definition}

\begin{proposition}\label{prop:scalar}
Let $p\ge7$ be regular, and let $m>0$ be prime to $p$. Then $p^{p+3}m\in S(C_p\times C_p)$ if and only if $\Ltwo(D)$ holds for some
principal assignment for $m$, and $p^{p+4}m\in S(C_p\times C_p)$ if and only if
$\Ltwo(D)$ or $\Lthree(D,m\bmod p)$ holds for some principal assignment
for $m$. The same criteria hold for $-p^{p+3}m$ and $-p^{p+4}m$, and
every principal assignment satisfying the relevant condition yields a
vector in $\Z^G$ with the corresponding determinant. Put
\[
 M'=\frac{(p+1)(p-3)}{2}-1.
\]
In both criteria it suffices to consider assignments that place at most
$h_K^{p+1}p^{M'}-1$ prime ideals above each rational prime $\ell$,
independently of $v_\ell(m)$; when $h_K=1$, at most $(p-1)M'$ prime
ideals above each $\ell$ suffice.
\end{proposition}

\begin{proof}
Suppose that $p^{p+3}m$ or $p^{p+4}m$ occurs. By
Lemma~\ref{lem:valuations}, after a linear change of coordinates the pair
$(v_p(s),e_\infty)$ is $(1,2)$ in the first case and $(2,2)$ or $(1,3)$ in
the second. Removing the powers of $\p$ from the character values
leaves principal cofactor ideals $I_\delta$.
By~\eqref{eq:factor}, $m$ is the product of their norms and the part of
$s$ prime to $p$. Since a prime ideal above $\ell$ has norm
$\ell^{\ord_p(\ell)}$, the cofactor ideals form a principal assignment.
Put $n_0=\prod_b\Nm(I_b)\,\Nm(I_\infty)$. Every assigned prime ideal has
norm $\ell^{\ord_p(\ell)}\equiv1\pmod p$, so $n_0\equiv1\pmod p$, and in
the case $(v_p(s),e_\infty)=(1,3)$ we have $s/p=m/n_0\equiv m\pmod p$. The changes of coordinates used in
Section~\ref{sec:systems} permute the directions and conjugate the
ideals, and the set of assignments is stable under both operations. By
Proposition~\ref{prop:placed}, the assignment satisfies $\Ltwo$ in the first case,
and $\Ltwo$ or $\Lthree$ in the second.

Conversely, take a principal assignment and define $n_0$ as above. Then
$n_0\mid m$ and $n_0\equiv1\pmod p$. If $\Ltwo$ holds, take
$s=pm/n_0$ for $p^{p+3}m$ or $s=p^2m/n_0$ for $p^{p+4}m$;
condition $\Ltwo$ does not involve $\tau$. If $\Lthree$ holds, take
$s=pm/n_0$, so that $s/p\equiv m\pmod p$. We choose generators of the
principal ideals $I_\delta$ and apply Proposition~\ref{prop:placed}, which
gives the desired determinant. Since $S(C_p\times C_p)=-S(C_p\times C_p)$, this proves the assertions about signs.

For the bound, record for each $I_\delta$ its ideal class, together with
its $(p-3)/2$ coordinates $c_j(I_\delta)$ with $j=3,5,\ldots,p-2$ if
$\delta\in\F$, and its $(p-5)/2$ coordinates $c_j(I_\infty)$ with
$j=3,5,\ldots,p-4$ if $\delta=\infty$; by Definition~\ref{def:affine-tests},
$\Ltwo$ and $\Lthree$ involve $d_{\infty,j}$ only for $j\le p-4$. These
classes and coordinates lie in the
finite abelian group
\begin{equation}\label{eq:state-group}
 \Omega=\Cl(R)^{p+1}\times\F^{M'},\qquad|\Omega|=h_K^{p+1}p^{M'}.
\end{equation}
Fix a rational prime $\ell$. One conjugate of $\mathfrak q_\ell$ placed
in one direction contributes an element of $\Omega$, and the part of an
assignment above $\ell$ is a word in these elements. A shortest word
with a given sum in $\Omega$ has no two equal prefix sums: otherwise the
nonempty subword between them sums to zero and could be deleted. Its
length is therefore at most $|\Omega|-1$. Replacing the part above
$\ell$ by such a shorter word preserves the ideal class of every
$I_\delta$ and the coordinates recorded above. It does not increase the
norm used, and it leaves the conditions $\Ltwo$ and $\Lthree$
unchanged. The factors of $m$ that are no longer used are absorbed
into $s$, and the residue of $s/p$ is
preserved because their norms are congruent to one modulo $p$. When
$h_K=1$, the contributed elements lie in the vector space $\F^{M'}$.
Choose, among these elements, a basis of their span. Every attainable
vector is then a combination of at most $M'$ basis vectors with coefficients in
$\{0,\ldots,p-1\}$, and is attained by placing at most $(p-1)M'$
conjugates of $\mathfrak q_\ell$. If the original assignment is shorter,
we keep it.
\end{proof}

\begin{definition}\label{def:cofactor-sets}
For a regular prime $p\ge7$, let $\mathcal C_{2,p}$ be the set of
positive integers $m$ prime to $p$ that have a principal assignment
satisfying $\Ltwo(D)$, and let $\mathcal C_{23,p}$ be the set of those
that have a principal assignment satisfying $\Ltwo(D)$ or
$\Lthree(D,m\bmod p)$.
\end{definition}

\begin{proof}[Proof of Theorem~B]
The values prime to $p$ form $\mathcal U_p$~\cite{DeSilvaPinner2014}, and
by Lemma~\ref{lem:valuations} every nonzero $p$-divisible value has
$p$-adic valuation at least $p+3$. Proposition~\ref{prop:scalar} describes the values
of valuation $p+3$ and $p+4$, and Theorem~A(a), with $r_p=1$, gives
$p^{p+5}\Z$. This proves~\eqref{eq:intro-whole-set}. For each $m$ there
are finitely many assignments, so whether $\pm p^{p+3}m$ or $\pm p^{p+4}m$
lies in $S(C_p\times C_p)$ is decided by a finite computation once
$\Cl(R)$ is known and generators of principal ideals can be computed,
and Proposition~\ref{prop:scalar} produces a vector for each such
determinant.
\end{proof}

\section{Proof of Theorem C}\label{sec:prime-cofactors}

If the cofactor is a rational prime $\ell\ne p$, at most one prime
ideal can occur among the $I_\delta$. With Lemma~\ref{lem:support} this
excludes prime cofactors at $p$-adic valuation $p+3$ when $p\ge11$
(Corollary~\ref{cor:prime-exclusion}). At $p$-adic valuation $p+4$, the test reduces to the coordinates of a single principal
prime ideal.

\begin{lemma}\label{lem:prime-support}
Let $p\ge7$ be regular and let $m>0$ be prime to $p$. If an assignment
for $m$ satisfies $\Ltwo$, then
\begin{equation}\label{eq:prime-support}
 \sum_{\substack{\ell\mid m\\ \ell^3\equiv1\ (\mathrm{mod}\ p)}}
 \Bigl\lfloor\frac{v_\ell(m)}{\ord_p(\ell)}\Bigr\rfloor\ \ge\ p-6.
\end{equation}
In particular, this inequality is necessary for
$\pm p^{p+3}m\in S(C_p\times C_p)$.
\end{lemma}

\begin{proof}
If $\Ltwo$ holds, then $D_3$ has degree six by
Definition~\ref{def:affine-tests},
so $D_3(b)\ne0$ for at least $p-6$ values of $b\in\F$. For each such
$b$, additivity of $c_3$ yields a prime ideal $\mathfrak q$ placed in
the direction $b$ with $c_3(\mathfrak q)\ne0$; by~\eqref{eq:frobenius},
the prime $\ell$ below $\mathfrak q$ satisfies $\ell^3\equiv1\pmod p$.
Each prime ideal of the assignment is placed in exactly one direction,
and at most $\lfloor v_\ell(m)/\ord_p(\ell)\rfloor$ of them lie above
$\ell$. This proves~\eqref{eq:prime-support}. The last assertion follows
from Proposition~\ref{prop:scalar}.
\end{proof}

\begin{proof}[Proof of Theorem~C(a)]
For $m=\ell$, the
left side of~\eqref{eq:prime-support} is at most one, whereas $p-6>1$,
so no assignment for $\ell$ satisfies $\Ltwo$. It remains to analyze
$\Lthree$.

An assignment for $\ell$ is empty unless $\ell\equiv1\pmod p$. In that
case it places at most one prime ideal $\mathfrak q$ above $\ell$ in a
single direction, and $\mathfrak q$ must then be principal. The empty
assignment would, by Proposition~\ref{prop:placed}, give unit character
values with $e_\infty=3$ and $s=p\ell$. Lemma~\ref{lem:parity} would
then give $p^2\mid s$, a contradiction.

Suppose first that $\mathfrak q$ is placed in a direction $b_0\in\F$, and let $\mathbf 1_{b_0}(b)=1-(b-b_0)^{p-1}$ be the indicator
function of $b_0$. Then $D_j=c_j(\mathfrak q)\mathbf 1_{b_0}$, $d_\infty=0$, and $\mathbf 1_{b_0}$ has degree $p-1$. Recall
that $q=0$. Let $j\ge3$ be odd and suppose that
$c_i(\mathfrak q)=0$ for every odd $i$ with $3\le i<j$. Then $H_b/T$ is
the product of an even series with constant coefficient one and a series
$1+D_jT^j+O(T^{j+1})$, so $h_{j+1}=D_j$. For
$j=3,5,\ldots,p-4$ the required degree of $h_{j+1}$ is at most
$j+1\le p-3<p-1$, and induction gives
$c_3(\mathfrak q)=\cdots=c_{p-4}(\mathfrak q)=0$.

The only even coefficient of $H_b$ left is
$h_{p-1}=c_{p-2}(\mathfrak q)\mathbf 1_{b_0}$. The series
$H_\infty/(h_{\infty,3}T^3)$ again has vanishing linear and odd
logarithmic coordinates, so $h_{\infty,p-1}=0$. Since
$\Nm(\mathfrak q)=\ell$, we have $s=p$ and $\tau=1$, and the
augmentation condition requires
\begin{equation}\label{eq:prime-finite-top}
 c_{p-2}(\mathfrak q)=-1.
\end{equation}
In particular this coordinate is nonzero.

If instead $\mathfrak q$ is placed in the exceptional direction, every $D_j$ is
zero. The expansions $H_b$ are then odd, and
by~\eqref{eq:degree} the exceptional coefficients of even degree up to
$p-3$ vanish. Put $a=h_{\infty,3}\ne0$. The series $H_\infty/(aT^3)$
has vanishing coefficients in the odd degrees $1,3,\ldots,p-6$, and so
does its logarithm. Hence
$c_3(\mathfrak q)=\cdots=c_{p-6}(\mathfrak q)=0$. The coefficient of
$T^{p-4}$ in $H_\infty/(aT^3)$, the largest degree below $p-3$, equals
$c_{p-4}(\mathfrak q)$, and the augmentation condition reads
\begin{equation}\label{eq:prime-exceptional-top}
 1+a\,c_{p-4}(\mathfrak q)=0.
\end{equation}
Hence $c_{p-4}(\mathfrak q)\ne0$. Together the two placements
show that~\eqref{eq:prime-criterion} is necessary.

Conversely, let $\ell\equiv1\pmod p$, and let $\mathfrak q$ be a
principal prime ideal above $\ell$ that
satisfies~\eqref{eq:prime-criterion}. Suppose first that
$\omega=c_{p-4}(\mathfrak q)$ is nonzero. Place $\mathfrak q$ in the
exceptional direction, put $a=-\omega^{-1}$, and prescribe
\begin{equation}\label{eq:prime-exceptional-data}
 H_b=T+ab^3T^3,\qquad H_\infty=aT^3+a\omega T^{p-1},\qquad s=p.
\end{equation}
Each $H_b/T$ has an even logarithm. The series $H_\infty/(aT^3)$ is
$1+\omega T^{p-4}$, and the only odd term of degree less than $p-3$ in
its logarithm is $\omega T^{p-4}$. The degree conditions~\eqref{eq:degree}
hold: $h_1=1$ and $h_3=ab^3$ have the required degrees, their
coefficients $[b^n]h_n$ equal $h_{\infty,n}$, and all other coefficients
through degree $p-2$ vanish. The augmentation condition holds because
$a\omega=-1$.

If instead $c_{p-4}(\mathfrak q)=0$ and $c_{p-2}(\mathfrak q)\ne0$,
replace $\mathfrak q$ by its conjugate under $\sigma_t$ with
$t=-c_{p-2}(\mathfrak q)$, whose coordinate $c_{p-2}$ is $-1$. Place
this conjugate in the direction $b_0=0$ and, for an arbitrary
$a\in\F^\times$, prescribe
\begin{equation}\label{eq:prime-finite-data}
 H_b=T+ab^3T^3-\mathbf 1_0(b)T^{p-1},\qquad H_\infty=aT^3,\qquad s=p.
\end{equation}
The logarithm of $H_b/T$ has no odd term below degree $p-2$, and its
coefficient of $T^{p-2}$ is $-\mathbf 1_0(b)$, because dividing
$\mathbf 1_0(b)T^{p-2}$ by $1+ab^3T^2$ changes it only in degrees at least $p$.
The series $H_\infty/(aT^3)$ is constant. The degree conditions~\eqref{eq:degree} hold as in~\eqref{eq:prime-exceptional-data}. The top
coefficients sum to $-1$ because $\sum_b\mathbf 1_0(b)=1$, so the
augmentation condition also holds.

In both cases the lift~\eqref{eq:affine-unit-recovery} produces
character values $f_\delta$ with the prescribed principal cofactor
ideals, and by Lemma~\ref{lem:lifting} there is $v\in\Z^G$ with augmentation
$p$ and $\Theta(v)=p\prod_\delta\Nm(f_\delta)$. Here $v_\p(f_b)=1$ for
$b\in\F$ and $v_\p(f_\infty)=3$, and the norms of the cofactor ideals
multiply to $\ell$. Hence $\Theta(v)=p^{p+4}\ell$, and
negating $v$ gives $-p^{p+4}\ell$.

Finally, $\sigma_t$ multiplies each $c_j(\mathfrak q)$ by $t^j\ne0$ and
preserves principality, and all prime ideals above a prime $\ell$
that splits completely are conjugate. Hence the criterion does not depend on the chosen ideal.
\end{proof}

Let $\Cl_{pR}(K)$ be the ray class group
of $K$ modulo $pR=\p^{p-1}$; the modulus has no infinite part, since $K$
is totally imaginary. \citet[Lemma~3.7, $k=p-1$]{Hoelscher2010} showed that this group is
isomorphic to $\Cl(R)\times\F^{(p-3)/2}$, and the following lemma gives an
isomorphism in terms of our coordinates. The
coordinates $c_j(I)$ extend additively to fractional ideals $I$ prime
to $\p$; equivalently, \eqref{eq:ideal-coordinates} may be applied to a
generator that is a unit at $\p$.

\begin{lemma}[Ray class coordinates]\label{lem:prime-ray-class}
For a regular prime $p\ge7$, the map
\begin{equation}\label{eq:prime-ray-isomorphism}
 \Cl_{pR}(K)\longrightarrow\Cl(R)\times\F^{(p-3)/2},\qquad
 [I]_{pR}\longmapsto\bigl([I],c_3(I),c_5(I),\ldots,c_{p-2}(I)\bigr),
\end{equation}
is an isomorphism. It is equivariant for $\sigma_t$, which acts as usual
on $\Cl(R)$ and by multiplication by $t^j$ on the coordinate $c_j$. In
particular, $\lvert\Cl_{pR}(K)\rvert=h_Kp^{(p-3)/2}$.
\end{lemma}

\begin{proof}
A principal ideal generated by an element congruent to one modulo $pR$
has vanishing coordinates $c_j$, so the map is well defined. Additivity
and equivariance follow from Lemma~\ref{lem:ideals}. If the image of
$[I]_{pR}$ is zero, then $I=zR$ and $c_j(z)=0$ for $j=3,5,\ldots,p-2$.
Lemma~\ref{lem:units} provides a global unit $u$ with the same
residue as $z$. Then $z/u\equiv1$ modulo $pR$ and $[I]_{pR}=0$, which
proves injectivity. The exact sequence~\cite[(3.1)]{Hoelscher2010}
\[
 R^\times\longrightarrow(R/pR)^\times\longrightarrow\Cl_{pR}(K)
 \longrightarrow\Cl(R)\longrightarrow0
\]
and the order $(p-1)p^{(p-1)/2}$ of $\Ehat$ (Lemma~\ref{lem:units}) give
$\lvert\Cl_{pR}(K)\rvert=h_K(p-1)p^{p-2}/\bigl((p-1)p^{(p-1)/2}\bigr)
=h_Kp^{(p-3)/2}$, which is the order of the target. Hence the map is an
isomorphism.
\end{proof}

\begin{lemma}\label{lem:ray-class-field}
Let $p\ge5$ be prime and let $\widetilde K$ be the ray class field of $K$
modulo $pR$. Then $\widetilde K/\Q$ is Galois, the Artin map identifies
$\Delta=\Gal(\widetilde K/K)$ with $\Cl_{pR}(K)$, and conjugation on the
abelian normal subgroup $\Delta$ of $\Gal(\widetilde K/\Q)$ factors
through $\Gal(K/\Q)$.
\end{lemma}

\begin{proof}
Class field theory identifies $\Delta$ with $\Cl_{pR}(K)$ through the
Artin map~\cite[Chapter~VI]{Neukirch1999}. The modulus $pR$ and the group
of principal ideals generated by elements congruent to one modulo $pR$
are stable under $\Gal(K/\Q)$. Hence every automorphism $\sigma$ of
$\overline{\Q}$ maps the kernel of the Artin map to itself, and
uniqueness in the class field correspondence gives
$\sigma(\widetilde K)=\widetilde K$. Thus $\widetilde K/\Q$ is Galois, and
conjugation on the abelian normal subgroup $\Delta$ factors through the
quotient by $\Delta$, which is $\Gal(K/\Q)$.
\end{proof}

\begin{proof}[Proof of Theorem~C(b)]
By Lemma~\ref{lem:ray-class-field}, the ray class field $\widetilde K$ of
$K$ modulo $pR$ is Galois over $\Q$, and by
Lemma~\ref{lem:prime-ray-class} its Galois group $\Gamma$ has order
$[K:\Q]\,|\Delta|=(p-1)h_Kp^{(p-3)/2}$. Conjugation on $\Delta$
corresponds, under the Artin map, to the action in
Lemma~\ref{lem:prime-ray-class}.

Let $\mathfrak S\subset\Delta$ consist of the classes whose image
under~\eqref{eq:prime-ray-isomorphism} has trivial component in
$\Cl(R)$, vanishing coordinates $c_3,\ldots,c_{p-6}$, and last two
coordinates $c_{p-4},c_{p-2}$ not both zero. This set has
$p^2-1$ elements. The Galois action preserves it and acts on the last pair by
$(c_{p-4},c_{p-2})\mapsto(t^{-3}c_{p-4},t^{-1}c_{p-2})$, so
$\mathfrak S$ is stable under conjugation in $\Gamma$. Every rational
prime $\ell\ne p$ is unramified in $\widetilde K$, and its Frobenius lies in
$\Delta$ exactly when $\ell$ splits completely in $K$, that is, when
$\ell\equiv1\pmod p$. In that case the Frobenius conjugacy class of
$\ell$ is the orbit of the Artin symbol of a prime $\mathfrak q\mid\ell$ of residue degree one. By Theorem~C(a), the primes $\ell\ne p$
with $\pm p^{p+4}\ell\in S(C_p\times C_p)$ are exactly those whose Frobenius lies in
$\mathfrak S$. The Chebotarev density theorem for a conjugation-stable
subset (\cite[Chapter~VII, Theorem~13.4]{Neukirch1999} for Dirichlet density,
\cite[Section~2.1, Theorem~1]{Serre1981} for natural density) shows that, among all
rational primes, these primes have natural density and Dirichlet
density both equal to
\[
 \frac{|\mathfrak S|}{|\Gamma|}=\frac{p^2-1}{(p-1)h_Kp^{(p-3)/2}}
 =\frac{p+1}{h_Kp^{(p-3)/2}}.\qedhere
\]
\end{proof}

\begin{remark}
The modulus $pR$ in Lemma~\ref{lem:prime-ray-class} cannot be replaced
by $\p^{p-2}$. The element $1+\varpi^{p-2}$ is congruent to one modulo
$\p^{p-2}$, but its coordinate $c_{p-2}$ modulo $pR$ is $-1$, so by
Lemma~\ref{lem:prime-ray-class} its principal ideal is nontrivial in
$\Cl_{pR}(K)$. The field $\widetilde K$ therefore has conductor
exactly $pR$~\cite[Chapter~VI]{Neukirch1999}.
\end{remark} 

\begin{lemma}\label{lem:norm-bound}
Let $\ell\ne p$ be a prime that splits completely in $K$, let $y\in R$
generate a prime ideal above $\ell$, and let $k$ be odd with
$3\le k\le p-2$. If $c_j(y)=0$ for every odd $j$ with $3\le j<k$, then
$\ell>p^k/2^{p-1}$.
\end{lemma}

\begin{proof}
Choose an integer $g$ with
$g\equiv c_1(y)\pmod p$ and put $y_0=\zeta^{-g}y$. Then $c_1(y_0)=0$ and
$c_j(y_0)=c_j(y)$ for $j\ge2$. Let $\overline{y_0}$ denote the complex
conjugate of $y_0$. Complex conjugation fixes residues modulo $\p$ and
replaces $T$ by $-T$, which gives $\overline{y_0}(0)=y_0(0)$ and
$c_j(\overline{y_0})=(-1)^jc_j(y_0)$. The residue of $y_0/\overline{y_0}$ modulo
$pR$ is therefore
\[
 \exp\Bigl(2\sum_{j\ \mathrm{odd}}c_j(y_0)T^j\Bigr)
 \equiv1\pmod{T^k},
\]
and $y_0-\overline{y_0}\in\p^{k}$. Since $\ell$ splits completely in $K$,
the ideal $yR$ differs from its complex conjugate, and
$y_0\ne\overline{y_0}$. Thus $y_0-\overline{y_0}$ is a nonzero element of
$\p^{k}$, and $p^{k}$ divides $\Nm(y_0-\overline{y_0})$. On the other
hand, $\sigma(\overline{y_0})=\overline{\sigma(y_0)}$ for every complex
embedding $\sigma$ of $K$, so $|\sigma(y_0-\overline{y_0})|\le2|\sigma(y_0)|$.
Taking the product over the $p-1$ embeddings, we obtain
$\Nm(y_0-\overline{y_0})\le2^{p-1}\Nm(y_0)=2^{p-1}\ell$. Hence
$p^{k}\le2^{p-1}\ell$, and equality is impossible because the left
side is odd.
\end{proof}

\begin{proof}[Proof of Theorem~C(c)]
Let $\ell\ne p$ be a prime with $\pm p^{p+4}\ell\in S(C_p\times C_p)$. By
Theorem~C(a), $\ell\equiv1\pmod p$ and a prime ideal
$\mathfrak q\mid\ell$ has a generator $y$ with
$c_3(y)=\cdots=c_{p-6}(y)=0$. Lemma~\ref{lem:norm-bound} with $k=p-4$
gives $\ell>p^{p-4}/2^{p-1}$.
\end{proof}

\begin{remark}\label{rem:finite-direction-bound}
Let $p\ge11$ be regular and suppose that $\pm p^{p+4}\ell\in S(C_p\times C_p)$
for a prime $\ell\ne p$. If a prime ideal $\mathfrak q$ above $\ell$
satisfies $c_{p-4}(\mathfrak q)=0$, then Theorem~C(a) gives
$c_{p-2}(\mathfrak q)\ne0$ and $c_3(\mathfrak q)=\cdots=c_{p-6}(\mathfrak q)=0$,
and Lemma~\ref{lem:norm-bound} with $k=p-2$, applied to a generator of
$\mathfrak q$, gives
\begin{equation}\label{eq:generic-norm-bound}
 \ell>\frac{p^{p-2}}{2^{p-1}}.
\end{equation}
\end{remark}

\begin{proposition}\label{prop:noncongruence}
Let $p\ge11$ be regular. For no positive integer $n$ does the set of
primes $\ell$ with $p^{p+4}\ell\in S(C_p\times C_p)$ differ from the set of primes in
a union of residue classes modulo $n$ by a set of primes of density zero
among all primes.
\end{proposition}

\begin{proof}
In the notation of the proof of Theorem~C(b),
let $\Delta_0\subset\Delta$ correspond
under~\eqref{eq:prime-ray-isomorphism} to the classes with trivial ideal
class and $c_3=\cdots=c_{p-6}=0$, so that $\mathfrak S=\Delta_0\smallsetminus\{1\}$.
An element $\rho\in\Gamma$ restricting to complex conjugation on $K$ acts
on $\Delta_0$ by $w\mapsto w^{-1}$, so $\rho w\rho^{-1}w^{-1}=w^{-2}$, and since $p$ is
odd every element of $\Delta_0$ is a commutator. Hence $\Delta_0$ acts trivially on
$\widetilde K\cap\Q(\zeta_n)$, which is abelian over $\Q$, and the pairs
$(w,1)$ with $w\in\mathfrak S$ form a conjugation-stable subset of
$\Gal(\widetilde K(\zeta_n)/\Q)$, viewed as the group of pairs in
$\Gamma\times\Gal(\Q(\zeta_n)/\Q)$ that agree on
$\widetilde K\cap\Q(\zeta_n)$. By the Chebotarev density theorem, the
primes with Frobenius in this subset and the primes that split
completely in $\widetilde K(\zeta_n)$ both have positive density among
all primes. All of
them are congruent to $1$ modulo $n$; the former satisfy the criterion of Theorem~C(a),
and the latter do not.
\end{proof}

\section{Proof of Theorem D}\label{sec:density}

Start with the character values of a determinant of valuation $p+3$.
We can choose $h_{\infty,p-1}$ to meet the augmentation
condition~\eqref{eq:constant} because the degree
conditions~\eqref{eq:degree} involve only the lower coefficients.
Replacing the cofactor ideals by prime ideals in the same ray classes
modulo $pR$ preserves the adjusted expansions. Trafford uses the
corresponding replacement modulo a conductor of an order in a
direct sum of number fields~\cite[Chapter~1, Lemma~1]{Trafford1992}.

We will need distinct rational prime factors from prescribed sets.
The following argument uses Tur\'an's second-moment method~\cite{Turan1934}.

\begin{lemma}\label{lem:many-factors}
Let $\mathcal P$ be a set of primes with
$\sum_{\ell\in\mathcal P}1/\ell=\infty$, and let $k\ge1$. The positive
integers with fewer than $k$ distinct prime factors in $\mathcal P$ form
a set of natural density zero.
\end{lemma}

\begin{proof}
Let $\omega_{\mathcal P}(n)$ be the number of distinct prime factors of
$n$ in $\mathcal P$, and put
$E_{\mathcal P}=E_{\mathcal P}(x)=\sum_{\ell\in\mathcal P,\,\ell\le x}1/\ell$.
Counting the multiples of each $\ell$ and of each product $\ell\ell'$ of
two distinct primes gives
\[
 \sum_{n\le x}\omega_{\mathcal P}(n)\ge xE_{\mathcal P}-x,\qquad
 \sum_{n\le x}\omega_{\mathcal P}(n)^2\le xE_{\mathcal P}^2+xE_{\mathcal P},
\]
and hence $\sum_{n\le x}(\omega_{\mathcal P}(n)-E_{\mathcal P})^2\le3xE_{\mathcal P}$.
If $E_{\mathcal P}>k$, every $n\le x$ with $\omega_{\mathcal P}(n)<k$
contributes more than $(E_{\mathcal P}-k)^2$ to this sum, so there are
fewer than $3xE_{\mathcal P}/(E_{\mathcal P}-k)^2$ such $n$. This is
$o(x)$, because $E_{\mathcal P}(x)\to\infty$.
\end{proof}

\begin{proof}[Proof of Theorem~D(a)]
By~\cite[Chapter~2, Lemma~13]{Trafford1992} (see
also~\cite[Theorem~2.1]{MossinghoffPinnerHeisenberg2024}) some
$v_0\in\Z^G$ has $v_p(\Theta(v_0))=p+3$. By Lemma~\ref{lem:valuations}, in
suitable coordinates its augmentation $s_0$ has $v_p(s_0)=1$ and its
character values satisfy $v_\p(f_b)=1$ for $b\in\F$ and $v_\p(f_\infty)=2$;
by Lemmas~\ref{lem:reconstruct} and~\ref{lem:moments}, their expansions
satisfy~\eqref{eq:degree} and~\eqref{eq:constant} with
$\tau_0=s_0/p\bmod p$. For $t\in\F$ put
\[
 g_t=f_\infty+(\tau_0-t)\varpi^{p-1},
\]
with $\tau_0-t$ represented by an integer. Then $v_\p(g_t)=2$, and since expansions are additive (they come from
the isomorphism $\p/\p^p\to T\F[T]/(T^p)$ of Section~\ref{sec:characters})
and $(1-\exp T)^{p-1}\equiv T^{p-1}\pmod{T^p}$, the expansion of $g_t$
is $H_\infty+(\tau_0-t)T^{p-1}$. Thus~\eqref{eq:degree} still holds, as it
involves only coefficients of degree at most $p-2$, and~\eqref{eq:constant}
holds with $\tau=t$. Put $y_b=f_b/\varpi$ and $y_{\infty,t}=g_t/\varpi^2$.

In the ideal-theoretic form of the ray class group
$\Cl_{pR}(K)$~\cite[VI, Proposition~1.9]{Neukirch1999}, the trivial class
consists of the ideals $\alpha R$ with $\alpha=\beta/\gamma$, where
$\beta,\gamma\in R$ are prime to $\p$ and $\beta\equiv\gamma\pmod{pR}$.
Hence, if
$y\in R$ is prime to $\p$ and an integral ideal $\mathfrak q$ lies in the
ray class of $yR$, then $\mathfrak q=y'R$ with $y'=y\alpha\in R$ and
$v_\p(y'-y)\ge p-1$, that is, $y'\equiv y\pmod{pR}$.

Let $\widetilde K$ be the ray class field of $K$ modulo $pR$; by
Lemma~\ref{lem:ray-class-field} it is Galois over $\Q$. For a ray class $\mathfrak c$,
the primes $\ell\equiv1\pmod p$ with a prime ideal above $\ell$ in
$\mathfrak c$ are those whose Frobenius conjugacy class in
$\Gal(\widetilde K/\Q)$ contains the Artin symbol of $\mathfrak c$. By the
Chebotarev density theorem~\cite[Chapter~VII, Theorem~13.4]{Neukirch1999}
they form a set $\mathcal P_{\mathfrak c}$ of positive Dirichlet density,
so $\sum_{\ell\in\mathcal P_{\mathfrak c}}1/\ell$ diverges.

Apply Lemma~\ref{lem:many-factors} with $k=p+1$ to the $2p$ sets
$\mathcal P_{\mathfrak c}$, where $\mathfrak c$ is the ray class of
$y_bR$ ($b\in\F$) or of $y_{\infty,t}R$ ($t\in\F$). Since
$S(C_p\times C_p)=-S(C_p\times C_p)$, we may take $m>0$. Outside a set of
natural density zero, an integer $m>0$ prime to $p$ has at least $p+1$
distinct prime factors in each of these sets. Fix such an $m$, put
$n=N-p-2\ge1$, and let $t=m\bmod p$ if $n=1$ and $t=0$ if $n\ge2$; write
$y_\infty=y_{\infty,t}$. Since $m$ has at least $p+1$ distinct prime
factors in each $\mathcal P_{\mathfrak c}$, we can choose, one direction at a time, distinct primes
$\ell_\delta\mid m$ ($\delta\in\mathbf P^1(\F)$) and prime ideals
$\mathfrak q_\delta\mid\ell_\delta$ in the ray class of $y_\delta R$. Write $\mathfrak q_\delta=y'_\delta R$ with
$y'_\delta\equiv y_\delta\pmod{pR}$. The elements $f'_b=\varpi y'_b$ and
$f'_\infty=\varpi^2y'_\infty$ are congruent to $f_b$ and $g_t$ modulo
$\varpi pR=\p^p$ and have the same valuations, so they have the same
expansions. Put $s=p^nm/\prod_\delta\ell_\delta$. Then $p\mid s$, and as
every $\ell_\delta$ is congruent to $1$ modulo $p$, we have
$s/p\equiv m\equiv t$ if $n=1$ and $s/p\equiv0=t$ if $n\ge2$.
By Lemma~\ref{lem:lifting} there is $v\in\Z^G$ with
\[
 \Theta(v)=s\prod_{\delta}\Nm(f'_\delta)
 =s\,p^{p+2}\prod_\delta\ell_\delta=p^Nm.\qedhere
\]
\end{proof}

\begin{proof}[Proof of Theorem~D(b)]
Let $N<e(p)$. Then $N\le2p-2$ by Theorem~A(c) and
Lemma~\ref{lem:admissible-range}, and Theorem~A(b) excludes $p^N\ell$
for every prime $\ell\not\equiv0,1\pmod p$, because then
$\ord_p(\ell)\ge2$. By the prime number theorem for arithmetic
progressions there are $\gg x/\log x$ such primes $\ell\le x$. Suppose
that membership of $p^Nm$ in $S(C_p\times C_p)$ depended only on $m$
modulo some $n\ge1$. The integers prime to $p$ in a residue class modulo
$n$ that contains one of them have positive density, so by part~(a)
each such class contains some $m$ with $p^Nm\in S(C_p\times C_p)$, and
then all its integers prime to $p$ have this property. This contradicts
the exclusion of $p^N$ by Theorem~A(b).
\end{proof}

\begin{proof}[Proof of Theorem~D(c)]
The values prime to $p$ form $\mathcal U_p$, and every nonzero
$p$-divisible value lies in $p^{p+3}\Z$, by Lemma~\ref{lem:valuations}
for $p\ge5$ and by~\cite[Theorem~6.1]{PinnerSmyth2020} for $p=3$. For
$p=3,5$ every multiple of $p^{p+3}$ is a
value~\cite[Theorem~6.1]{PinnerSmyth2020}, \cite{Panraksa2026Order25}.
For $p\ge7$, the set $S(C_p\times C_p)$ contains $p^{e(p)}\Z$, and by
part~(a), for each $p+3\le N<e(p)$, the integers $p^Nm$ with $p\nmid m$
that are not values form a set of density zero. In both cases
$S(C_p\times C_p)$ differs from the disjoint union
$\mathcal U_p\cup p^{p+3}\Z$ by a set of density zero. As
$(\Z/p^2\Z)^\times$ is cyclic of order $p(p-1)$, the set $\mathcal U_p$
is a union of $p-1$ residue classes modulo $p^2$, and the density
of $S(C_p\times C_p)$ is $(p-1)/p^2+p^{-(p+3)}$.
\end{proof}

\section{Examples and questions}\label{sec:examples}

\subsection*{The exponent \texorpdfstring{$e(p)$}{e(p)} at irregular primes}

For the irregular primes in Table~\ref{tab:tail}, Theorem~A(d) gives
$e(p)=p+3+2r_p$. To compute $r_p$, try successive positive integers $r$
and discard any that divides an $i\in\Irr$ with $i\le r(p-1)/(2r+1)$.
This bound lies between $(p-1)/3$ and $(p-1)/2$. For each of $103$, $131$
and $283$ there is one irregular pair, with $i=12$, $11$ and $10$,
respectively. Each satisfies $3i\le p-1$, so the first $r$ that fails to
divide $i$ gives $r_p$. The three values are $r_{103}=5$, $r_{131}=2$
and $r_{283}=3$. The largest value for $p<400$ is $r_{103}=5$.

\begin{table}[ht]
\centering
\caption{The exponent $e(p)=p+3+2r_p$ at the irregular primes
$p<400$.}\label{tab:tail}
\begin{tabular}{rlrr@{\qquad}rlrr}
\toprule
$p$ & $2i$ & $r_p$ & $e(p)-p$ & $p$ & $2i$ & $r_p$ & $e(p)-p$\\
\midrule
37 & 32 & 1 & 5 & 263 & 100 & 3 & 9\\
59 & 44 & 1 & 5 & 271 & 84 & 4 & 11\\
67 & 58 & 1 & 5 & 283 & 20 & 3 & 9\\
101 & 68 & 1 & 5 & 293 & 156 & 4 & 11\\
103 & 24 & 5 & 13 & 307 & 88 & 3 & 9\\
131 & 22 & 2 & 7 & 311 & 292 & 1 & 5\\
149 & 130 & 1 & 5 & 347 & 280 & 1 & 5\\
157 & 62, 110 & 2 & 7 & 353 & 186, 300 & 2 & 7\\
233 & 84 & 4 & 11 & 379 & 100, 174 & 4 & 11\\
257 & 164 & 3 & 9 & 389 & 200 & 3 & 9\\
\bottomrule
\end{tabular}
\end{table}

For $p=103$ and $r=1,2,3,4$, if $P$ is chosen so that the logarithmic
coordinate of degree $24$ vanishes, circular units realize the residues required
by~\eqref{eq:tail-data}, and $\deg h_{25}\le25$ is the only degree
condition that fails: $h_{25}$ has degree $(2r+1)\cdot12/r$, that is,
$36$, $30$, $28$ and $27$, as in the proof of
Lemma~\ref{lem:support}. The reconstruction congruences~\eqref{eq:reconstruction-congruence}
for the resulting projections then fail exactly at the cells
$(i,j)\in\F^2$ with $ij\ne0$. For $r=1$ the same happens at $p=131$ and $p=157$,
where $h_{23}$ and $h_{63}$ have degrees $33$ and $93$, respectively.

Table~\ref{tab:rp} counts the primes $7\le p<200000$ with each
value of $r_p$; $7073$ of these $17981$ primes are irregular. Suppose that
the events $p\mid B_{2i}$ were independent with probability $1/p$, as in
Siegel's heuristic for the number of irregular pairs
(compare~\cite[Table~2]{HartHarveyOng2017}). As $p\to\infty$, the
numbers $i/p$ with $i\in\Irr$ would form a Poisson process of
intensity one on $(0,1/2)$, with residues modulo each fixed integer
equidistributed, and $r$ would be inadmissible exactly when some
irregular $i$ divisible by $r$ has $i/p\le r/(2r+1)$. Inclusion--exclusion
gives the predicted counts in Table~\ref{tab:rp}; for instance, $r_p=1$
has probability $e^{-1/3}$. The value $r_p=6$ is much rarer than $5$ and
$7$: if $k$ is not a prime power, then
$\operatorname{lcm}(1,\ldots,k-1)=\operatorname{lcm}(1,\ldots,k)$, so a
single irregular $i$ that makes $1,\ldots,k-1$ inadmissible also makes
$k$ inadmissible. In the model $r_p$ exceeds any given $k$ with positive
probability, since one irregular $i\le p/3$ divisible by
$\operatorname{lcm}(1,\ldots,k)$ makes $1,\ldots,k$ inadmissible. The
largest value below $200000$ is $r_p=13$, at $p=172243$, with irregular
pairs $(172243,3630)$ and $(172243,75600)$, whereas
Lemma~\ref{lem:admissible-range} bounds $r_p$ only by about $3p/8$.

\begin{table}[ht]
\centering
\caption{Counts of primes $7\le p<200{,}000$ by $r_p$, with
predictions from the independence model.}\label{tab:rp}
\begin{tabular}{@{}rrr@{}}
\toprule
$r_p$ & Observed & Predicted\\
\midrule
1 & 12872 & 12883.9\\
2 & 2240 & 2260.0\\
3 & 1686 & 1671.4\\
4 & 511 & 527.0\\
5 & 482 & 448.2\\
6 & 33 & 28.9\\
7 & 118 & 122.6\\
8 & 13 & 17.4\\
9 & 12 & 12.8\\
$\ge10$ & 14 & 8.6\\
\bottomrule
\end{tabular}
\end{table}

We also checked the stronger inequality~\eqref{eq:minimum-test} for
every prime $11\le p<10^7$. Hart, Harvey and Ong's verification of
Vandiver's conjecture~\cite{HartHarveyOng2017} gives $\Irru=\Irr$
throughout this range. For each prime in their table we halved the
listed irregular indices $2i$. For each $a\ge1$ we took the least
$(2a+1)i/a$ over those $i$ divisible by $a$, taking the minimum of an
empty set to be $+\infty$. The first $a$ for which this minimum exceeds
$p-1$ gives $r$. For each $a<r$ we then compared its minimum with
$p-2(r-a)-1$. A sieve checked the prime list through the cutoff.
All $664\,575$ primes, including $261\,871$ irregular primes, satisfy
the criterion. Thus the least positive $p$-divisible determinant is
$p^{e(p)}$ throughout this finite range.

\subsection*{Prime cofactors for \texorpdfstring{$p=11$}{p=11} and
\texorpdfstring{$p=13$}{p=13}}

Since $\Q(\zeta_{11})$ and $\Q(\zeta_{13})$ have class number one~\cite[pp.~412 and~421]{Washington1997},
the criterion of Theorem~C(a) involves only the odd coordinates of a
generator of a prime ideal above each $\ell\equiv1\pmod p$. A search over
all primes $\ell\equiv1\pmod p$ below $3\cdot10^7$ for $p=11$ and below
$4\cdot10^7$ for $p=13$ shows that the least primes $\ell$ with
$\pm p^{p+4}\ell\in S(C_p\times C_p)$ are $77419$ and $6285241$,
respectively.

\begin{table}[ht]
\centering
\caption{The number of primes $\ell\le x$ with
$p^{p+4}\ell\in S(C_p\times C_p)$, compared with the count $\pi(x)(p+1)/p^{(p-3)/2}$ suggested by
Theorem~C(b) with $h_K=1$, where $\pi(x)$ is the number of primes up to $x$.}\label{tab:primes}
\begin{tabular}{rrrr}
\toprule
$p$ & $x$ & number of $\ell\le x$ & $\pi(x)(p+1)/p^{(p-3)/2}$\\
\midrule
11 & $10^5$ & 2 & 7.9\\
11 & $10^6$ & 54 & 64.3\\
11 & $10^7$ & 537 & 544.7\\
11 & $3\cdot10^7$ & 1510 & 1522.7\\
13 & $10^7$ & 3 & 25.1\\
13 & $4\cdot10^7$ & 31 & 91.8\\
\bottomrule
\end{tabular}
\end{table}

For $p=11$, Theorem~C(c) gives
$\ell>11^7/2^{10}\approx1.9\cdot10^4$. A generator of a prime ideal
above $77419$ is
\begin{equation}\label{eq:generator-certificate}
 \begin{split}
 Y(X)&=-2-2X-3X^2-2X^3-X^4-2X^5-X^6-3X^7,\\
 y&=Y(\zeta_{11}).
 \end{split}
\end{equation}
Its norm $\operatorname{Res}(\Phi_{11},Y)$ is $77419$, so $yR$ is a
prime ideal of that norm. Since $Y(1)=-16\equiv6\pmod{11}$, substitution of
$X=\exp T$ and expansion of $\log(Y(\exp T)/6)$ modulo $T^{10}$ gives
\[
 (c_3(y),c_5(y),c_7(y),c_9(y))=(0,0,10,9).
\]
These coordinates satisfy~\eqref{eq:prime-criterion}, and the
construction~\eqref{eq:prime-exceptional-data} applies with $\omega=10$
and $a=1$. The vector $v$ with $\Theta(v)=11^{15}\cdot77419$ in the
repository (see Data availability) comes from the same construction
applied to the conjugate $\sigma_4(y)$, whose coordinates are
$(0,0,6,5)$, with $\omega=6$. The last equation of
$\Lthree$ is then $1+6a=0$ in $\mathbf F_{11}$, so $a=9$ and
\[
 H_b=T+9b^3T^3,\qquad H_\infty=9T^3-T^{10}.
\]

For $p=13$ the counts in Table~\ref{tab:primes} fall well short of the
prediction. Theorem~C(c) gives $\ell>13^9/2^{12}\approx2.59\cdot10^6$ for
every prime $\ell$ with $p^{p+4}\ell\in S(C_p\times C_p)$, and by
Remark~\ref{rem:finite-direction-bound} every such prime with $c_9=0$
satisfies $\ell>13^{11}/2^{12}>4\cdot10^7$. Hence every prime counted for
$p=13$ in Table~\ref{tab:primes} has $c_9\ne0$.

\subsection*{Questions}

At an irregular prime the residues of the global units have vanishing
coordinates $c_{2i}$ for $i\in\Irru$, so the cofactor ideals impose
even conditions as well as odd ones, and these are not affine in the
even coefficients. In addition, $p$ divides the class number, and the
coordinates of nonprincipal ideals cannot be defined
by~\eqref{eq:ideal-coordinates}. Theorem~D shows that the missing
cofactors have density zero but does not identify them. Which cofactors
are missing at the valuations $p+3$ and $p+4$ when $p$ is irregular?
Which values $p^Nm$ with $p+5\le N<e(p)$ are missing when $m$ is prime to
$p$ and divisible by the norm of some prime ideal of $\Z[\zeta_p]$?

What is the order of magnitude of the number of exceptions $m\le x$ in
Theorem~D(b)? Theorem~A(b) excludes every squarefree $m$ none of whose
prime factors is congruent to $0$ or $1$ modulo $p$, so this number
is at least the number of such $m\le x$. Does $p^{e(p)}$ give the least positive
$p$-divisible determinant for every prime $p\ge11$? Is
$e(p)-p$ unbounded, as the model suggests under Vandiver's conjecture?

In higher rank, the inclusion
$p^{2p^{n-1}}S(C_p^{n-1})\subseteq S(C_p^n)$
\cite[Theorem~2.3]{MossinghoffPinner2026} transfers the bound in
Theorem~A to every $n\ge3$. Since $p+3+2r_p\le2p-2$, this improves on the
inclusions obtained from $p^{2p+2}\Z\subseteq S(C_p\times C_p)$. What is the least
exponent $e$ with $p^e\Z\subseteq S(C_p^n)$, and how does it depend on
the irregular pairs of $p$?

\section*{Data availability}

The programs, inputs and outputs for the exponent census and the
prime-cofactor computations are available at
\url{https://github.com/chatchawanpan-dev/cpxcp-group-determinant-computations},
commit \href{https://github.com/chatchawanpan-dev/cpxcp-group-determinant-computations/tree/ae9e36800676e82a8aada90d7a204e50d2e7e328}{\texttt{ae9e36800676}}.
This version includes the
\href{https://github.com/chatchawanpan-dev/cpxcp-group-determinant-computations/blob/ae9e36800676e82a8aada90d7a204e50d2e7e328/density_and_census/python/rp_model.py}{inclusion--exclusion program}
used for Table~\ref{tab:rp}. Its
\href{https://github.com/chatchawanpan-dev/cpxcp-group-determinant-computations/tree/ae9e36800676e82a8aada90d7a204e50d2e7e328/minimum}{\texttt{minimum} directory}
contains the checker for~\eqref{eq:minimum-test}, the exact input rows
for $11\le p<10^7$, their provenance and SHA-256 checksums, reproduction
commands and the expected output. The input rows come from
\texttt{indices.txt} in Hart, Harvey and Ong's data
archive~\cite{HartHarveyOng2017},
\url{https://web.maths.unsw.edu.au/~davidharvey/research/twobillion.tar.bz2}.
The range checker and its expected output also accompany the sources of
this paper.

\section*{Declaration of generative AI and AI-assisted technologies in the manuscript preparation process}

During the preparation of this work, ChatGPT 6 Astra (OpenAI), Claude
Opus 5.5 (Anthropic) and OpenAI Codex were used to draft and check proofs, to
write and check computer calculations, to check references and to revise
the text. The author takes full responsibility for the content.

\bibliographystyle{plainnat}
\setlength{\bibsep}{3pt}
\interlinepenalty=10000
\bibliography{references}
\end{document}